\pdfoutput=1
\documentclass[a4paper,11pt]{article}

\usepackage{amsmath,amsfonts,amssymb,caption,graphicx,amstext}
\usepackage{ushort}
\usepackage{makeidx}
\usepackage{float}
\usepackage{accents}
\usepackage{color}
\usepackage{framed}
\usepackage{versions}
\usepackage{xcolor}
\usepackage{emptypage}
\usepackage{wrapfig}
\usepackage[T1]{fontenc}
\usepackage{titlesec}
\usepackage{fancyhdr}
\usepackage{extramarks}
\usepackage[numbers]{natbib}
\usepackage[bookmarks]{hyperref}
\usepackage{bbm}
\usepackage{xcolor}
\usepackage{mathrsfs}
\usepackage{lmodern}
\usepackage{amsthm}
\usepackage{listings}
\usepackage{textcomp}
\usepackage{comment}
\usepackage{import}
\usepackage{pdfpages}
\usepackage{xparse}
\usepackage{algorithm}
\usepackage{algorithmic}
\usepackage{siunitx}
\usepackage{tikz} 
\usepackage{pgfplots}
\pgfplotsset{compat=newest} 
\usepgfplotslibrary{external,groupplots,dateplot, units} 
\newcommand{\mbold}[1]{\boldsymbol{#1}}

\newcommand{\expec}[1]{\mathbbm{E}\left[#1\right]}

\NewDocumentCommand{\evalat}{sO{\big}mm}{%
  \IfBooleanTF{#1}
   {\mleft. #3 \mright|_{#4}}
   {#3#2|_{#4}}%
}

\newtheorem{theorem}{Theorem}
\newtheorem{corollary}{Corollary}
\newtheorem{lemma}{Lemma}

\hypersetup{
    colorlinks=true,   	
    linkcolor=red,      
    citecolor = [rgb]{0 0.7 0},   	
    filecolor=magenta, 	
    urlcolor=blue
}

\title{Generalised shot noise representations of stochastic systems
driven by non-Gaussian L\'{e}vy processes }

\author
{
    Marcos Tapia Costa
    \thanks{Department of Engineering, Trumpington Street,
	University of Cambridge, Cambridge, CB2 1PZ, UK. Email: \texttt{\href{mailto:mt773@cantab.ac.uk} {mt773@cantab.ac.uk}}}
\and
        Ioannis Kontoyiannis 
    \thanks{Statistical Laboratory, DPMMS,
	University of Cambridge,
	Centre for Mathematical Sciences,
        Wilberforce Road,
	Cambridge CB3 0WB, UK.
                Email: \texttt{\href{mailto:yiannis@maths.cam.ac.uk}%
			{yiannis@maths.cam.ac.uk}}.
	I.K.\ was supported in part by the Hellenic Foundation for Research 
	and Innovation (H.F.R.I.) under the ``First Call for H.F.R.I. Research 
	Projects to support Faculty members and Researchers and the 
	procurement of high-cost research equipment grant,'' project 
	number 1034.
        }
\and
	Simon Godsill
    \footnotemark[1]
    \thanks{Email: \texttt{\href{mailto:sjg30@cam.ac.uk}%
			{sjg30@cam.ac.uk}}.
        }
}

\date{\today}

\begin{document}

\maketitle

\begin{abstract}
We consider the problem of obtaining effective representations for the solutions of linear, vector-valued stochastic differential equations (SDEs) driven by non-Gaussian pure-jump L\'{e}vy processes, and we show how such representations lead to efficient simulation methods. The processes considered constitute 
a broad class of models that find application across the physical and biological sciences, mathematics, finance and engineering. 
Motivated by important relevant problems in statistical inference, we derive 
new, generalised shot-noise simulation methods whenever a normal variance-mean  (NVM) mixture representation exists for the driving L\'{e}vy process, including the generalised hyperbolic, normal-Gamma, and normal tempered stable cases. Simple, explicit conditions are identified for the convergence of the residual of a truncated shot-noise representation to a Brownian motion in the case of the pure L\'{e}vy process, and to a Brownian-driven SDE in the case of the L\'{e}vy-driven SDE. These results provide Gaussian approximations to the small jumps of the process under the NVM representation. The resulting representations are of particular importance in state inference and parameter estimation for L\'{e}vy-driven SDE models, since the resulting conditionally Gaussian structures can be readily incorporated into latent variable inference methods such as Markov chain Monte Carlo (MCMC), Expectation-Maximisation (EM), and sequential Monte Carlo.     

\end{abstract}

\section{Introduction}
    Lévy processes are commonly employed in the study of asset returns, derivative pricing models, and in the prediction of high-frequency trading returns~\cite{LevyInFinance,VarianceGammaOptionPricing,HighFrequencyReturns}. In derivative pricing in particular, the Black-Scholes model specifies the dynamics of a financial market which holds derivative instruments such as options, futures, and swaps~\cite{BlackScholesModel}. Yet, the assumption that option prices are modelled by a stationary log-Gaussian process often fails to be validated by empirical data~\cite{BlackScholesModel}, motivating research into the relaxation of this assumption. Lévy processes such as 
    the normal tempered stable~\cite[Section~9]{NormalModifiedStable}, the generalised hyperbolic~\cite{GIGinFinance}, or the variance-Gamma~\cite{VarianceGammaOptionPricing,InterestRatesVG} can be used to model the variance, or volatility, of the security as a stochastic process.
    
    In biology, the generalised inverse Gaussian process has been used to model spike train activity from neurons~\cite{GIGNeuralActivity}, or to model the likelihood of extreme meteorological events with high economic and social costs~\cite{GIGHydrologic}. The tempered stable process has also been extensively studied in relation to Lévy flight models~\cite{TSLevyFlight} in physics, and the variance-Gamma has been used to analyse continuous trait evolution such as population growth~\cite{Landis2013PhylogeneticAU,RUSSO2009521}, or in the analysis of nitrate concentration in the eastern US~\cite{LevyNitrate}. More recently, the leptokurtic properties of non-Gaussian Lévy processes have been exploited in the field of object tracking~\cite{LevyObjectTracking} in order to capture abrupt changes in the kinematic state of a moving object, which in turn has applications to animal tracking and the analysis of hunting patterns.

    
    We are here concerned with the analysis of different representations of L\'{e}vy-driven stochastic processes of the form,
    \begin{equation*}
        d\boldsymbol{X}(t) = \mbold{A}\boldsymbol{X}(t)dt + \mbold{h}dW(t), \ \boldsymbol{X}(t) \in \mathbbm{R}^{P},
    \end{equation*}
where $\mbold{A}$ is a $P\times P$ matrix, $\mbold{h}\in\mathbbm{R}^P$, and $(W(t))$ is a one-dimensional non-Gaussian L\'{e}vy process.
Based on infinite series representations for \(\left(W(t)\right)\), tractable and conditionally Gaussian models can be developed for simulation and inference purposes, see, for example,~\cite{Lemke_Godsill_2015,LevyStateSpaceModel,LevyObjectTracking}, where such models and inference procedures were developed for $\alpha$-stable L\'{e}vy processes driving linear stochastic differential equations (SDEs). 
Here by contrast we study variance-mean mixture L\'{e}vy processes~\cite{LevyInFinance}, including (but not limited to) the normal-Gamma, normal tempered stable, and generalised hyperbolic processes. In particular, we provide functional central limit theorem-like results for the convergence of the residual terms when series representations of such processes are truncated at a finite level, and for the corresponding Gaussian convergence of the above SDE's residuals when it too is truncated in a corresponding fashion.  
The results are linked in spirit to the study of~\cite{SmallJumps} although posed in terms of a different truncation of small jumps. 

Lévy processes have been studied extensively via the general theory of stochastic processes with stationary, independent increments. Samorodnitsky and Taqqu~\cite{StableNonGaussianProcesses}, Küchler and Tappe~\cite{TSDistAndProcesses}, and Barndorff-Nielsen and Shephard~\cite{LevyInFinance,NormalModifiedStable,LevyBasics} consider important special cases, including the tempered stable, inverse Gaussian, and  \(\alpha\)-stable processes. The relevant properties established are then used in the definition and analysis of  broader classes of Lévy processes, such as the normal variance-mean (NVM) mixture processes~\cite{NormalModifiedStable}, which are formed by the subordination of Brownian motion to a non-negative process. Samorodnitsky and Taqqu also examine stochastic integrals with respect to \(\alpha\)-stable random noise, providing a series representation closely related to the series representation of the \(\alpha\)-stable process. Barndorff-Nielsen extends the series representation of \(\alpha\)-stable stochastic integrals to ones driven by a non-negative Lévy process~\cite[Section~8]{LevyInFinance}, and Rosinski studies an analogous representation for symmetric Gaussian mixture processes~\cite{GenShotNoiseStochIntegral}. 

The breadth and depth of the study of Lévy processes have allowed for numerous studies into their applications in modelling rare-event phenomena e.g. in finance~\cite{LevyInFinance,HighFrequencyReturns,FinancialModellingCont}, physics~\cite{TSLevyFlight}, and biology~\cite{GIGNeuralActivity,GIGHydrologic,LevyNitrate,LevyObjectTracking}.

Efficient methods for simulating Lévy processes have been critical in successfully bridging the gap between theory and application. Muller and Box outline a method of generating normal random variates~\cite{BoxMuellerMethod}, and Chambers, Mallows and Stuck study the generation of \(\alpha\)-stable random variables, both of which are summarised in work by Barndorff-Nielsen~\cite[Sections~2.3-2.4]{LevyBasics}, and Samorodnitsky and Taqqu~\cite{StableNonGaussianProcesses}. While simulation from finite-activity processes is straightforward~\cite{FinancialModellingCont}, exact simulation from infinite-activity processes, such as those considered in this work, is impossible, due to the presence of an infinite number of small jumps in any finite time interval. Khinchin~\cite{KhintchineSim}, Bondesson~\cite{BondessonSim}, and Fergusson and Klass~\cite{FergussonKlassSim} outline the inverse Lévy method for simulating jumps from infinite-activity processes with non-negative increments,~and Rosinski~\cite{GeneralisedShotNoise}  {develops  generalised shot noise methods with thinning and rejection sampling for processes with Lévy densities that cannot be simulated directly using approaches such as the inverse L\'{e}vy method.} In~\cite{GIGLevy}, Godsill and Kindap propose novel algorithms for simulating from the infinite activity generalised inverse Gaussian process by combining previous approaches, paving the way for direct simulation of generalised hyperbolic jumps.

While most simulation methods in principle require the computation of an infinite sum, truncation of the series is of course required in practice. Asmussen and Rosinski~\cite{SmallJumps} analyse the effect of strict truncation of the process jumps on the representation of the `residual process', namely, the truncation error. They provide necessary and sufficient conditions for the convergence of the residual to a diffusion process, and several authors \cite{DeligiannidisTruncation,BoundsSmallJumps} provide bounds on the rate of convergence of the residual for specific classes of Lévy processes. Further work~\cite{TotalVariationGaussian,TotalVariationGaussian2} has also examined the Gaussian representation of strictly small jumps when studying SDEs. {In the setting of SDEs driven by L\'{e}vy processes the work of \cite{Higa_Tankov_2010} provides a promising approach in which the process is simulated by ODE solvers between large jumps of the process, leading to an alternative and general methodology for nonlinear SDEs.}  In the context of the generalised shot noise representation of more complex processes, Rosinski studies instead random Poisson truncations of the epochs of the process. This methodology is applied by Godsill et al.~\cite{LevyStateSpaceModel} to the \(\alpha\)-stable process, where they provide verification of the convergence of the \(\alpha\)-stable residual to a Brownian motion. Yet, the analysis there is limited to the $\alpha$-stable normal mixture process where the mean and standard deviation are proportional to an underlying stable process, and it does not consider the much broader class of NVM mixture processes, where the mean and variance are proportional to the subordinator.

A framework for performing inference on state-space models described by linear Lévy-driven SDEs is presented in~\cite{LevyStateSpaceModel}. The problem of inference on linear Gaussian state-space models has largely been solved through the development of the Kalman filter~\cite{KalmanFilterOriginal}. Variants such as the extended Kalman filter or the unscented Kalman filter have also been introduced in order to handle non-linearities~\cite{NonLinearGaussianInference}. Inference on non-Gaussian models has been accomplished primarily through Bayesian modelling combined sequential Monte Carlo methods (outlined, e.g. in~\cite{OverviewGodsill,KantasSequentialStateSpace}), and this is exploited in~\cite{LevyStateSpaceModel} through the use of the Rao-Blackwellised particle filters to perform inference for a conditionally Gaussian state-space model driven by Lévy noise. The algorithm there uses the Kalman filter on the linear conditionally Gaussian part, and it employs particle filtering to infer the distribution of the hidden process driven by the \(\alpha\)-stable noise. The main aim of this paper is to provide the theoretical foundations with which to extend the `Lévy State Space Model'~\cite{LevyStateSpaceModel} to a broader class of processes.
         
    \section{Background: Lévy processes and series representations}
        \subsection{Lévy processes}
        A Lévy process, \((X(t))=(X(t)\;;\;t\geq 0\), is a real-valued, infinitely divisible stochastic process
        with stationary and independent increments. The log-characteristic function, commonly known as the {\em characteristic exponent} (CE) or \textit{cumulant function}, of any Lévy process is given by~\cite{JumpTypeProcesses},
        \begin{equation}
            K(t;\theta) :=\ln\mathbbm{E}\left(e^{i\theta X(t)}\right) = ti\theta a - t\frac{1}{2}b^{2}\theta^{2} + t\int_{\mathbbm{R}^*}\left[e^{i\theta x} -1 - \mathbbm{1}(|x|<1)ix\theta\right]Q(dx),
            \label{eq:LevyKhinRep}
        \end{equation}
where \(\mathbbm{R}^{*} := \mathbbm{R}\backslash \{0\}\).  The term $\mathbbm{1}(|x|<1)ix\theta$ is a centering term that ensures convergence of the CE for processes for which  $\int_{|x|<1}xQ(dx)$ is divergent, though it can be omitted for processes with finite first absolute moment. The Lévy triplet \((a, b^{2}, Q)\) uniquely defines the Lévy process~\cite{JumpTypeProcesses}, with \(a \in \mathbbm{R}\),
\(b \in [0,\infty)\), and where the \textit{Lévy measure}, \(Q(dx)\), is a Poisson process intensity measure defining the distribution of jumps in the Lévy process, satisfying,
           \begin{equation}
            \int_{\mathbbm{R}^*}\left(1 \wedge x^{2}\right)Q(dx) < \infty,
            \label{eq:LevyMeasureCondition}
        \end{equation}
        where $(a\wedge b)$ denotes the minimum value of $a$ and $b$. See, e.g.,~\cite[Section~2.3]{LevyBasics} for more details.
     
        \par
        \subsection{Subordinators\label{section:SubordinatorBackground}}
        A subordinator \((Z(t))_{t\geq0}\) is a particular case of a Lévy processes with non-negative increments, such that its paths are a.s.\ increasing~\cite{CompensatorBath}. The Lévy measure of any subordinator, \(Q_{Z}(dz)\), satisfies~\cite{LevyBasics},
        \begin{equation}
            \int_{(0,\infty)} (1\wedge z)Q_{Z}(dz) < \infty.
            \label{eq:SubordinatorRequirement}
        \end{equation}
        Observe that this is a stricter condition than that required for general Lévy processes in (\ref{eq:LevyMeasureCondition}).
        Consequently, the Lévy triplet of a subordinator has \(b^{2} = 0\) and no centering term is required~\cite{CompensatorBath}. The current work will consider only subordinators without drift, thus the CE for such a subordinator $(Z(t))$ will always
        be of the form:
        \begin{equation}
           K_Z(t;\theta):= \ln\mathbbm{E}[e^{i\theta Z(t)}] = t\int_{(0,\infty)}\left(e^{i\theta z} - 1\right)Q_{Z}(dz). \label{eq:SubordinatorCumulantFunction}
        \end{equation}
        The mean and variance of $Z(t)$, when they exist, may be obtained for all $t$ from the L{\'{e}}vy measure:
        $$
            \mathbbm{E}[Z(t)] = t\int_{(0,\infty)}z Q_Z(dz), \quad
            \mathbbm{Var}[Z(t)] = t\int_{(0,\infty)}z^{2}Q_{Z}(dz).
        $$
        
        Most but not all of the processes we consider here will have finite first and second moments because their Lévy densities decay exponentially for large jump sizes; see Sections~3.1.4 and~3.3.2 of~\cite{LevyBasics}. We will also be concerned here with so-called \textit{infinite activity} subordinators, which exhibit infinitely many jumps in any finite time interval: $Q_Z((0,\infty)) = \infty$. Combined with~(\ref{eq:LevyMeasureCondition}), this implies the presence of infinitely many \textit{small} jumps (\(|x|<1\)) within finite time intervals.
        
        \subsection{Normal Variance-Mean (NVM) processes\label{section:GeneralNormalVarianceBackground}}
        A \textit{normal variance-mean} (NVM) process is defined as \textit{time-deformed} Brownian motion, 
        where jumps in a subordinator process $(Z(t))$ drive random time deformations of an independent Brownian motion $(B(t))$
        as follows~\cite{LevyBasics}:
        \begin{equation}
            X(t) = \mu t + \mu_{W}Z(t) + \sigma_{W}B(Z(t)),  \quad t\geq 0,\; \mu, \mu_{W} \in \mathbbm{R}, \ \sigma_{W} \in (0,\infty).
            \label{eq:GeneralNormalVarianceMeanMixture}
        \end{equation}
        Here $(B(t))$ is a standard one-dimensional Brownian motion and $(Z(t))$ is a subordinator process as in the previous section.
         We limit attention to the case $\mu=0$ without loss of generality. The parameter \(\mu_{W}\) models the skewness of the jump distribution, with a fully symmetric process for \(\mu_{W} = 0\). 
 The specification of \(\left(Z(t)\right)\), coupled with the choice of $\mu_{W}$ and $\sigma_{W}$, allow for a broad family of heavy-tailed and skewed processes to be implemented.

                We can express the Lévy measure $Q$ of any such process in terms of its subordinator's Lévy measure \(Q_Z\),
        \begin{equation}
            Q(dx) = \int_{(0,\infty)}\mathcal{N}(dx;\mu_{W}z,\sigma_{W}^{2}z)Q_{Z}(dz),
            \label{eq:GeneralNormalVarianceMeanMeasure2}
        \end{equation}
        where ${\mathcal N}(\cdot;\mu,\sigma^2)$ denotes the Gaussian law with mean $\mu$ and variance $\sigma^2$. 
        Finally, if \(K_Z(t;\theta)\) is the Lévy-Khintchine exponent of \(Z(t)\) in (\ref{eq:SubordinatorCumulantFunction}), then the CE for the subordinated process is given by~\cite[Section~4.4]{FinancialModellingCont},
       {\begin{align}
            K_{X}(t;\theta) = K_Z\left(t; {{\mu_{W}\theta +i \frac{1}{2}\sigma_{W}^{2}\theta^{2}}}\right) &= t\int_{(0,\infty)}\left[e^{\left(i\mu_{W}\theta - \frac{1}{2}\sigma_{W}^{2}\theta^{2}\right)z}-1\right]Q_{Z}(dz),
            \label{eq:cumulantNVMGeneral}
        \end{align}
        from which we obtain the mean and variance directly in terms of the moments of the subordinator. Specifically, for $t\geq 0,$
        \begin{align}
            &\mathbbm{E}[X(t)] = \mu_{W}\int_{(0,\infty)}zQ_{Z}(dz) = t\mu_{W}\mathbbm{E}\left[Z(1)\right], 
            \label{eq:NVMMeanSub}\\
            &\mathbbm{Var}[X(t)] = t\mu_{W}^{2}\int_{(0,\infty)}z^{2}Q_{Z}(dz) + t\sigma_{W}^{2}\int_{0}^{\infty}zQ_{Z}(dz) = t\mu_{W}^{2}\mathbbm{Var}[Z(1)] + t\sigma_{W}^{2}\mathbbm{E}[Z(1)], 
            \label{eq:NVMVarianceSub}
        \end{align}
        when these expectations exist.
        
        We will consider examples based on the \textit{normal-Gamma, 
        normal tempered stable} and \textit{generalised hyperbolic} processes, which are obtained via  Brownian motion subordinated to the Gamma, tempered stable, and generalised inverse Gaussian processes, respectively, with Lévy measures shown in (\ref{eq:GMeasure}), (\ref{eq:TSMeasure}), (\ref{eq:GIGMeasure}), respectively,
        \begin{align}
            &Q_{Z}(dz) = \nu z^{-1}\exp\left(-\frac{1}{2}\gamma^{2}z\right)dz,\label{eq:GMeasure}\\
            &Q_{Z}(dz) = Az^{-1-\kappa}\exp\left(-\frac{1}{2}\gamma^{\frac{1}{\kappa}}z\right)dz,\label{eq:TSMeasure}\\
            &Q_{Z}(dz) = z^{-1}\exp \left(-\frac{\gamma^{2}}{2}z\right)\left[\max(0, \lambda)+\frac{2}{\pi^{2}}\int_{0}^{\infty}\frac{1}{y\left|H_{|\lambda|}(y)\right|^{2}} \exp \left( -\frac{zy^{2}}{2\delta^{2}}\right) dy\right] dz.\label{eq:GIGMeasure}
        \end{align}
        Appropriate ranges for the different parameters are specified in Section \ref{section:ExampleCases}.
        
        \subsection{Generalised shot noise representation \label{section:BackgroundGeneralisedShotNoise}}
        The shot noise representation of a Lévy process is fundamental to the simulation
        methods in~\cite{LevyStateSpaceModel},~\cite{GeneralisedShotNoise},~\cite{GIGLevy}. Adopting the perspective of viewing a Lévy process as a point process defined on \([0, T] \times \mathbbm{R}^{d}\), Rosinski~\cite{GeneralisedShotNoise} considers the equivalent representation, 
        \begin{equation}
            N = \sum_{i=1}^{\infty}\delta_{V_{i}, H(Z_{i}, U_{i})},
            \label{eq:GenShotNoisePointprocess}
        \end{equation}
        where \(\delta_x\) denotes the Dirac measure at $x$, and \(\left\{V_{i}\right\}\) is a sequence of independent and identically distributed (iid) uniforms, \(V_{i} \sim U[0,T]\), independent of \(\left\{Z_{i}, U_{i}\right\}\). Intuitively, \(\{V_{i}\}\) represent the arrival times of the jumps \(X_{i} = H(Z_{i}, U_{i})\) in the process, and $\{Z_i\}$, a non-increasing set of jump sizes drawn from the subordinator process.

                For the NVM L\'{e}vy process, $H(\cdot)$ is related to the time-domain representation of the process 
                in~(\ref{eq:GeneralNormalVarianceMeanMixture}) via:
        \begin{equation}
            H(Z_{i}, U_{i}) = \mu_{W}Z_{i} + \sigma_{W}\sqrt{Z_i}U_{i}, \quad \ U_{i} \overset{iid}{\sim} \mathcal{N}(0,1).
            \label{eq:JumpFunctionNormalVarianceMean}
        \end{equation}
        The ordered jumps $\{Z_i\}$ are typically simulated through generation of the ordered epochs $\{\Gamma_i\}$ of a standard Poisson point process, obtained by the partial sum of exponential random variables, \(e_{i} \overset{iid}{\sim} {\rm Exp}(1)\):
        $\Gamma_{i} = \Gamma_{i-1} + e_{i}.$
        Then we obtain the $i^{th}$ ordered subordinator jump through a function \(Z_{i} = h(\Gamma_{i})\) as,
        \begin{equation*}
            Z(t) = \sum_{i=1}^{\infty}h(\Gamma_{i})\mathbbm{1}(V_{i} \leq t),
        \end{equation*}
        where the map \(h(\cdot)\) can be expressed in terms of the upper tail of $Q_Z$~\cite[Section~3.4.1]{LevyBasics} as,
        $h(\Gamma_{i}) = \inf \{z \in \mathbb{R} : Q_Z([z,\infty)) < \Gamma_i\},$ 
        which may or may not be available in closed form, depending on the particular choice of \(Q_Z\). If not available, then rejection sampling or thinning methods can be employed, as in~\cite{GIGLevy}, which leads to suitable algorithms for all of the models considered here. 
    
        Substituting (\ref{eq:GenShotNoisePointprocess}) into the Lévy-Khintchine representation (\ref{eq:LevyKhinRep}), the Lévy process can be expressed as the convergent infinite series~\cite{GeneralisedShotNoise},
        \begin{equation}
            X(t) = \sum_{i=1}^{\infty}H(Z_i, U_{i})\mathbbm{1}(V_{i} \leq t) -tb_{i},
            \label{eq:GeneralisedShotNoiseProcess}
        \end{equation}
        where \(b_{i}\) is a compensator or centering term~\cite{GeneralisedShotNoise} that ensures the convergence of the series. 
         A Lévy process need not be compensated if and only if,
        \begin{equation*}
            \int_{\mathbbm{R}^*}(1 \wedge |x|)Q(dx) < \infty, 
        \end{equation*}
        see~\cite{CompensatorBath,MSCOxford}.
        In view of~(\ref{eq:SubordinatorRequirement}), any subordinator satisfies this condition, and hence so does the corresponding NVM processes; see Appendix~\ref{app:subordinator}. Therefore, none of the Lévy processes studied in this paper requires compensating terms, and we take \(b_{i} = 0 \) for all $i$ throughout.
 
\section{Random truncation of subordinator}\label{section:RandomTruncationMethod}
    Consider the problem of simulating an NVM L\'{e}vy process via the shot noise representation in the previous section.
    Suppose we have access to a simulation algorithm for subordinator generation, which is capable of producing a non-increasing set of random jumps $\{Z_i\}$, $i=1,2,\ldots$, and associated uniformly distributed random jump times $\{V_i\}$.  
    Exact simulation for infinite activity processes via~(\ref{eq:GeneralisedShotNoiseProcess}) is impossible
    with finite computational resources, so simulation of a truncated process is typically implemented. In this paper, we consider the effect of random Poisson truncations on the jumps of the subordinator process 
    as in~\cite{LevyStateSpaceModel,GeneralisedShotNoise},
\begin{align*}
    \hat{X}_{\epsilon}(t) & = \sum_{i: Z_i\geq {\epsilon}}H(Z_{i}, U_{i})\mathbbm{1}(V_{i} \leq t),\\
    X_{\epsilon}(t)     & = {X}(t) - \hat{X}_{\epsilon}(t) = \sum_{i: Z_i < \epsilon}H(Z_{i}, U_{i})\mathbbm{1}(V_{i} \leq t), 
\end{align*}
where, as before, $H$ is defined by $H(Z_{i}, U_{i}) = \mu_{W}Z_i + \sigma_{W}\sqrt{Z_i}U_{i}$, 
with $U_{i}$ being iid $\mathcal{N}(0,1).$
Then $\hat{X}_\epsilon(t)$ is the process $X(t)$ with subordinator jumps  truncated at level $\epsilon$, and $X_{\epsilon}(t)$ is the remainder, corresponding to the small jumps.  There intuitively exists a trade-off between the computational complexity of computing $\hat{X}_\epsilon(t)$, how accurately it approximates the true process $X(t)$, and whether a Gaussian approximation to $X_{\epsilon}(t)$ is valid.

The left-hand plots in Figures~\ref{fig:GammaPathsSimVer},~\ref{fig:TSPathsSimVer}, and~\ref{fig:GIGPathsSimVer} show sample paths from the {\em truncated} versions of the Gamma (with parameters $\gamma=\sqrt{2}$ and $\nu=2$), tempered stable
(with parameters $\kappa=1/2$, $\gamma=1.35$ and $\delta=1$), and generalised inverse Gaussian
(with parameters $\delta=1/3$, $\gamma=\sqrt{2}$ and $\lambda=0.2$)  {subordinators}, using the above methodology with $\epsilon=10^{-10}$, as described in~\cite{GIGLevy}. See Section~\ref{section:ExampleCases} for the precise definitions of these processes.
The corresponding right-hand plots compare the empirical distributions of \(N=10^5\) truncated process values at time \(t=1\), with random variates from the theoretical exact (not truncated) marginal distribution of each of these truncated processes at time \(t=1\). 
These are truncated at very low values of $\epsilon$ that might lead to infeasibly large computational burden in practical use. This motivates our justification of Gaussian approximations to the residuals in the following sections, when larger values of $\epsilon$ are used for computational reasons.

\begin{figure}[ht!]
    \begin{centering}
        \scalebox{0.5}{\includegraphics{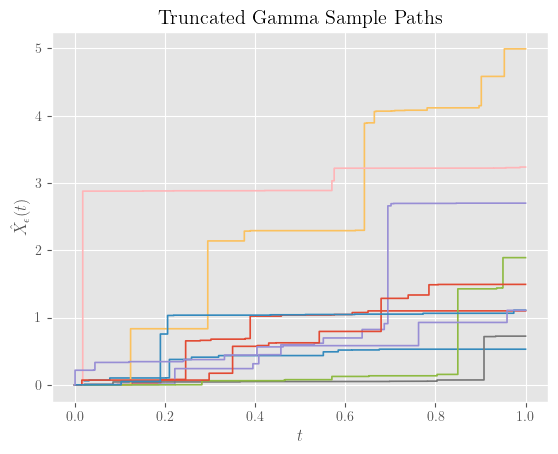}\includegraphics{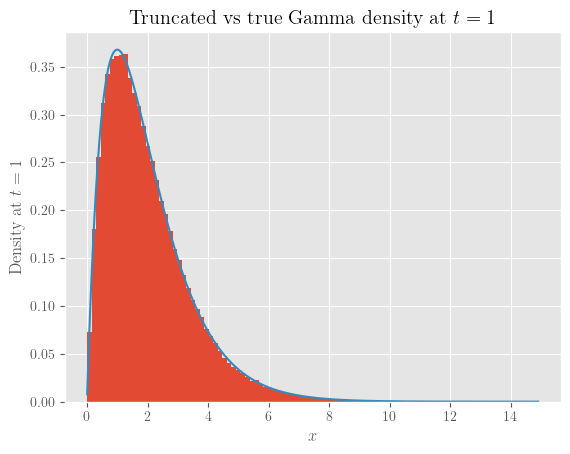}}
\caption{Left: ten sample paths from a truncated gamma process; right: histogram of $N=10^5$ process values at \(t=1\). Both generated with \(\epsilon = 10^{-10}\). The solid line is the true density of the original process at time $t=1$.\label{fig:GammaPathsSimVer}
}
\end{centering}
\end{figure}

\begin{figure}[ht!]
    \begin{centering}
    \scalebox{0.5}{\includegraphics{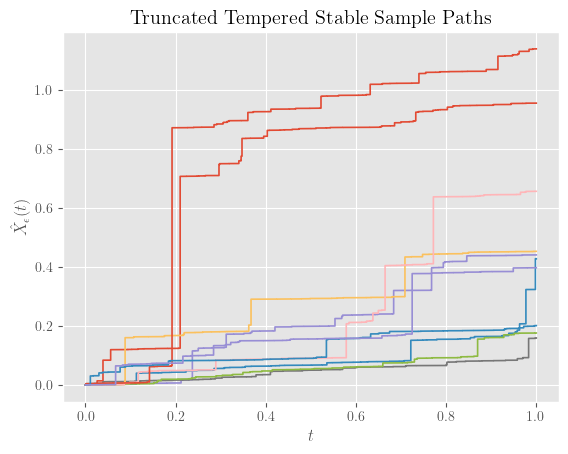}\includegraphics{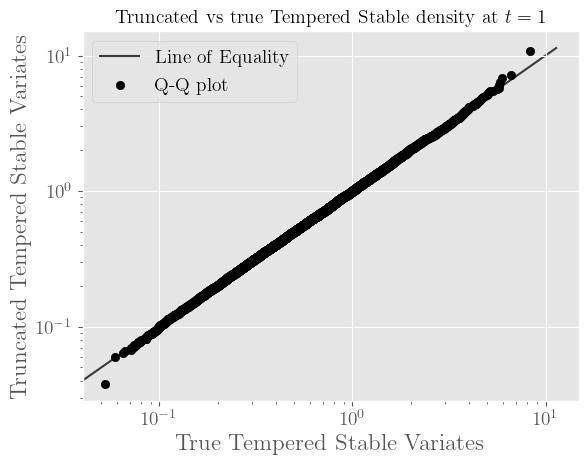}}
    \caption{Left: ten sample paths from a truncated tempered stable process; right: Q-Q plot of $N=10^5$ truncated process values at \(t=1\)
    versus $N=10^5$ samples from the true distribution of the process at $t=1$. Both generated with \(\epsilon = 10^{-10}\).}\label{fig:TSPathsSimVer}
    \end{centering}
\end{figure}

\begin{figure}[ht!]
    \begin{centering}
    \scalebox{0.5}{\includegraphics{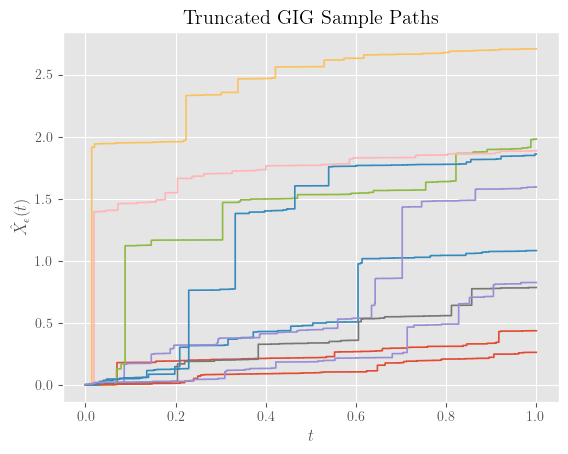}\includegraphics{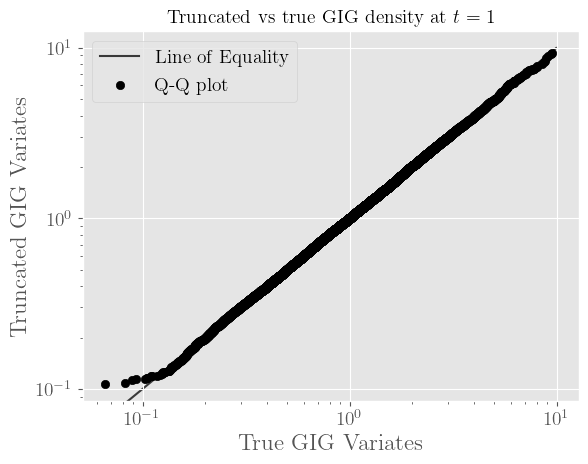}}
    \caption{Left: ten sample paths from a truncated generalised inverse Gaussian process; right: Q-Q plot of $N=4\times 10^4$ truncated process values at \(t=1\) versus $N=4\times10^4$ samples from the true distribution of the process at $t=1$. Both generated with \(\epsilon = 10^{-10}\). \label{fig:GIGPathsSimVer}}
    \end{centering}
\end{figure}

\newpage

\section{Gaussian process convergence}\label{section:Results}
\subsection{Preliminaries}
Consider a subordinator $(Z(t))$ with L\'{e}vy measure $Q_Z$. The corresponding process with jumps truncated at $\epsilon$ is denoted $(Z_\epsilon(t))$ and has L\'{e}vy measure 
$Q_{Z_\epsilon}(B)=Q_Z(B\cap (0,\epsilon])$ for all Borel sets $B$. The corresponding residual NVM process,
denoted $(X_\epsilon(t))$ has L\'{e}vy measure,
\begin{equation}
Q_{X_\epsilon}(dx)=\int_{0}^{\epsilon}{\cal {N}}(dx|\mu_{W} z, \sigma_{W}^2z)Q_Z(dz),\label{eq:Q_X_eps}
\end{equation}
and its moments can be obtained from (\ref{eq:NVMMeanSub}) and (\ref{eq:NVMVarianceSub}):
\[
\mathbbm{E}[X_\epsilon(t)]=t\int_{-\infty}^\infty x Q_{X_\epsilon}(dx)=t\mu_{W} M^{(1)}_{Z_\epsilon},
\]
and, 
\begin{equation}
    t\sigma_\epsilon^2=\mathbbm{Var} \left[X_\epsilon(t)\right]=t\int_{-\infty}^\infty x^2 Q_{X_\epsilon}(dx)=t\Big[\mu_{W}^2M^{(2)}_{Z_\epsilon}+\sigma_{W}^2M^{(1)}_{Z_\epsilon}\Big],
    \label{sigma_eps}
\end{equation}
where,
\begin{equation}
    M^{(n)}_{Z_\epsilon}=\int_0^\epsilon z^nQ_{Z_\epsilon}(dz)<\infty, \quad n\geq 1.
    \label{eq:Z_moments}
\end{equation}
Note that these moments are all well defined and finite since $Q_{Z_\epsilon}$ satisfies~(\ref{eq:SubordinatorRequirement}) 
and for all $n\geq 1$ we have $\lim_{\epsilon\rightarrow0}M_{Z_{\epsilon}}^{(n)} = 0$. 
Note also that, for any $0< \epsilon \leq 1$ and $m<n$,
\begin{equation}M^{(n)}_{Z_\epsilon}\leq  \epsilon^{n-m} M^{(m)}_{Z_\epsilon}, \quad n\geq 1,\label{M_condition}
\end{equation}
since $x^n\leq \epsilon^{n-m} x^m$ for $0<x\leq \epsilon$.
In particular this implies that,
\begin{equation*}\sigma_\epsilon^2=\mu_{W}^2M^{(2)}_{Z_\epsilon}+\sigma_{W}^2M^{(1)}_{Z_\epsilon}\leq
M^{(1)}_{Z_\epsilon}(\mu_{W}^2\epsilon+\sigma_{W}^2)\overset{\epsilon\rightarrow 0}{\rightarrow} 0, 
\end{equation*}
and that,
\begin{equation}
\frac{{\sigma_\epsilon}^4}{M^{(2)}_{Z_\epsilon}}=\frac{(\mu_{W}^2M^{(2)}_{Z_\epsilon}+\sigma_{W}^2M^{(1)}_{Z_\epsilon})^2}{M^{(2)}_{Z_\epsilon}}\xrightarrow{\epsilon \rightarrow 0} \sigma_{W}^4\lim_{\epsilon\rightarrow 0}\frac{{M^{(1)}_{Z_\epsilon}}^2}{M^{(2)}_{Z_\epsilon}},\label{eq:M_1_M_2}
\end{equation}
{whenever the limit $\lim_{\epsilon\rightarrow 0}\frac{{M^{(1)}_{Z_\epsilon}}^2}{M^{(2)}_{Z_\epsilon}}$ exists.}

\subsection{Gaussianity of the Limit of \(X_{\epsilon}\)}\label{section:SufficientCondition}

The following is our first main result: It gives conditions under which the residual process
corresponding to an NVM Lévy process with jumps truncated at level $\epsilon$,
converges to Brownian motion as $\epsilon\to 0$. An analogous one-dimensional result for $\alpha$-stable 
processes is established, along with corresponding convergence bounds, in~\cite{CharacteristicLevyIntegral}.

\begin{theorem}
\label{theorem:NVMConvergenceGaussian}
Consider a truncated NVM L\'{e}vy process $X_{\epsilon}=(X_{\epsilon}(t))$ with Lévy measure as in~{\em (\ref{eq:Q_X_eps})}. Let the standardised process $Y_\epsilon$ be defined as,
\(Y_\epsilon(t)=(X_\epsilon(t)-\mathbbm{E}[X_\epsilon(t)])/\sigma_\epsilon\), $t\geq 0$.
If,
\begin{align}
    \lim_{\epsilon \rightarrow 0}\frac{M_{Z_\epsilon}^{(2)}}{{M_{Z_\epsilon}^{(1)}}^2} = 0, \label{eq:NecessaryAndSufficient}
\end{align}
then $Y_\epsilon$ converges weakly to a standard Brownian motion
$W=(W(t))$ in $D[0,1]$ {with the topology of uniform convergence}.
\end{theorem}

\begin{proof}
The main content of the proof is
establishing the claim that $X_\epsilon(1)$, properly standardised, converges to a Gaussian
under~(\ref{eq:NecessaryAndSufficient}). 
For that,
it suffices to show that the CE of 
the standardised version of $X_\epsilon(1)$ converges pointwise to $-u^2/2$.
The CE of $X_\epsilon(1)$ is given as in~(\ref{eq:cumulantNVMGeneral}), by,
\begin{align*}
    \phi^\epsilon_{X}(u)=\int_{\mathbbm{R}^*} (\exp(iux)-1)Q_X^\epsilon(dx)=\int_0^\epsilon\left[\exp\Big(iu\mu_{W}z-\frac{1}{2}{u^2}\sigma_{W}^2z\Big)-1\right]Q_Z(dz),
\end{align*}
where $Q_X^\epsilon(dx)$ is the L\'{e}vy measure for $X_\epsilon$, having the same structure as (\ref{eq:Q_X_eps}).
The CE for $\tilde{X}_\epsilon(t)={X}_\epsilon(t)-\mathbbm{E}[{X}_\epsilon(t)]$, is then given by,
\begin{align*}
    \phi_{\tilde{X}}^{\epsilon}(u) = \phi_{X}^{\epsilon}(u) - iu\int_{\mathbbm{R}^*}xQ_{X}^{\epsilon}(dx) =  \int_0^\epsilon\left[\exp\left(iu\mu_{W}z-\frac{1}{2}u^2\sigma_{W}^2z\right)-1-iu\mu_{W}z\right]Q_{Z}(dz).
\end{align*}
Note that $X_\epsilon$ has finite mean, $\int_{\mathbbm{R}^*}x Q_X^\epsilon(dx)=\mu_{W}M_{Z\epsilon}^{(1)}<\infty$, by the finiteness of $M_{Z\epsilon}^{(1)}$ in~(\ref{eq:Z_moments}).
Now scaling the centred process, $\tilde{X}_{\epsilon}$, to have unit variance at $t=1$, i.e., letting,
$Y_\epsilon(t)=\tilde{X}_\epsilon(t)/\sigma_\epsilon,$
the CE for $Y_\epsilon(1)$ becomes, 
\begin{align*}
    \phi^\epsilon_Y(u) = \phi^\epsilon_{\tilde{X}}(u/\sigma_\epsilon)&= \int_{0}^{\epsilon}\left[\exp\left(v\frac{z}{\sigma_\epsilon^2}\right)-1-iu\mu_{W}\frac{z}{\sigma_\epsilon}\right]Q_Z(dz)\\
    &=\int_0^\epsilon\left[\exp\left(v\frac{z}{\sigma_\epsilon^2}\right)-1-v\frac{z}{\sigma_\epsilon^{2}}-\frac{1}{2}u^2\sigma_{W}^2\frac{z}{\sigma_\epsilon^2}\right]Q_Z(dz)\\
    &= \int_{0}^{\epsilon}\left[\exp\left(v\frac{z}{\sigma_\epsilon^2}\right)-1-v\frac{z}{\sigma_\epsilon^{2}}\right]Q_Z(dz)-\frac{1}{2}u^2\sigma_{W}^2\frac{M^{(1)}_{Z_\epsilon}}{\sigma_\epsilon^2}\\
    &=\varphi_\epsilon(u)-\frac{1}{2}u^2\psi_\epsilon,
\end{align*}
where $v={iu\mu_{W}}{\sigma_\epsilon}-{u^2\sigma_{W}^2}/2$, $\varphi_\epsilon(u)=\int_{0}^{\epsilon}[\exp(v\frac{z}{\sigma_\epsilon^2})-1-v\frac{z}{\sigma_\epsilon^{2}}]Q_Z(dz)$ and $\psi_\epsilon=\sigma_{W}^2\frac{M^{(1)}_{Z_\epsilon}}{\sigma_\epsilon^2}$.

Consider the difference between $\exp(\phi^\epsilon_Y(u))$ and the CF of the standard normal:
\begin{align*}
    e_\epsilon(u)&:=\exp\left[\varphi_\epsilon(u)-\frac{1}{2}u^2\psi_\epsilon\right]-\exp\left(-\frac{u^2}{2}\right)\\
    &=\exp\left(-\frac{u^2}{2}\right)\left\{\exp\left[\varphi_\epsilon(u)+\frac{1}{2}u^2(1-\psi_\epsilon)\right]-1\right\}.
\end{align*}
First we establish bounds on $\varphi_\epsilon$ and $1-\psi_\epsilon$. For $1-\psi_\epsilon$, from~(\ref{sigma_eps}) and~(\ref{M_condition}) we have,
$$1-\psi_\epsilon=\frac{\mu_{W}^2M^{(2)}_{Z_\epsilon}}{{\sigma_\epsilon}^2}=\frac{\mu_{W}^2M^{(2)}_{Z_\epsilon}}{\mu_{W}^2M^{(2)}_{Z_\epsilon}+\sigma_{W}^2M^{(1)}_{Z_\epsilon}}=\frac{\mu_{W}^2}{\mu_{W}^2+\sigma_{W}^2M^{(1)}_{Z_\epsilon}/M^{(2)}_{Z_\epsilon}},
$$
therefore, since clearly $1-\psi_\epsilon\geq 0$,
\begin{equation}
0\leq 1-\psi_\epsilon
\leq \frac{\mu_{W}^2}{\mu_{W}^2+\sigma_{W}^2/\epsilon}\xrightarrow[]{\epsilon\rightarrow 0}0.
\label{eq:phi_b_bound}
\end{equation}
For $\varphi_\epsilon$, first recall that for any complex $z$ with negative real part,
\begin{align}
    \Big|\exp(z)-\sum_{i=0}^{n}z^n/n!\Big|\leq |z|^{n+1}/(n+1)!, \label{eq:BoundExponential}
\end{align}
and so,
\begin{align}
    |\varphi_\epsilon(u)|&=\left|\int_0^\epsilon(\exp(vz/\sigma_\epsilon^2)-1-vz/\sigma_\epsilon^2)Q(dz)\right|\label{eq:shorts}\\
    &\leq 
    \int_0^\epsilon \left|(\exp(vz/\sigma_\epsilon^2)-1-vz/\sigma_\epsilon^2)\right|Q(dz)\nonumber\\
    &\leq \int_0^\epsilon \frac{1}{2}|v|^2\frac{z^2}{\sigma_\epsilon^4}Q(dz)\nonumber\\
    &=\frac{1}{2}|v|^2 \frac{M^{(2)}_{Z_\epsilon}}{\sigma_\epsilon^4}.\label{eq:O(V_sq)} 
\end{align}
The final integral is well defined by~(\ref{eq:Z_moments}).
Since \(|v|^2 = {u^2\mu_{W}^2}{\sigma_\epsilon^2}+{u^4\sigma_{W}^4}/4\), using (\ref{eq:M_1_M_2}) and (\ref{eq:phi_b_bound}),
\begin{align*}
    \lim_{\epsilon\rightarrow 0}|v|^2\frac{M^{(2)}_{Z_\epsilon}}{\sigma_\epsilon^4}&= \lim_{\epsilon\rightarrow 0}\left[u^{2}\frac{\mu_{W}^{2}M_{Z_{\epsilon}}^{(2)}}{\sigma_{\epsilon}^{2}}+\frac{1}{4}u^{4}\sigma_{W}^{4}\frac{M_{Z_{\epsilon}}^{(2)}}{\sigma_{\epsilon}^{4}}\right]\\
    &=\frac{1}{4}u^{4}\sigma_{W}^4\lim_{\epsilon\rightarrow 0}\frac{M^{(2)}_{Z_\epsilon}}{(\mu_{W}^2M^{(2)}_{Z_\epsilon}+\sigma_{W}^2M^{(1)}_{Z_\epsilon})^2}\\
    &=\frac{1}{4}u^4\lim_{\epsilon\rightarrow 0}
        \frac{M^{(2)}_{Z_\epsilon}}{(M^{(1)}_{Z_\epsilon})^2},
\end{align*}
where we used the fact that \(M^{(2)}_{Z_\epsilon}/M^{(1)}_{Z_\epsilon} \leq \epsilon \rightarrow 0\) from (\ref{M_condition}).
Combining the bounds in~(\ref{eq:phi_b_bound}), (\ref{eq:O(V_sq)}):
\begin{align*}
    \Big|\varphi_\epsilon(u)+\frac{1}{2}u^2(1-\psi_\epsilon)\Big|&\leq \frac{1}{2}|v|^2 \frac{M^{(2)}_{Z_\epsilon}}{\sigma_\epsilon^4}+\frac{1}{2}u^2\frac{\epsilon\mu_{W}^2}{\epsilon\mu_{W}^2+\sigma_{W}^2}\xrightarrow{\epsilon\rightarrow 0} \frac{u^4}{8}\lim_{\epsilon\rightarrow 0}    \frac{M^{(2)}_{Z_\epsilon}}{{M^{(1)}_{Z_\epsilon}}^2}.
\end{align*}
This bound is finite for any $|u|<\infty$ by the properties of 
$\sigma_\epsilon$ and $M^{(2)}_{Z_\epsilon}$, and hence, under~(\ref{eq:NecessaryAndSufficient}): 
\begin{align*}
    |e_\epsilon(u)|
    &=\left|\exp\left(-\frac{u^2}{2}\right)\left\{\exp\left[\varphi_\epsilon(u)+\frac{1}{2}u^2(1-\psi_\epsilon)\right]-1\right\}\right|\\
    &\leq \exp\left(-\frac{u^2}{2}\right)\left(\exp\left|\varphi_\epsilon(u)+\frac{1}{2}u^2(1-\psi_\epsilon)\right|-1\right) \xrightarrow[]{\epsilon\rightarrow 0} 0.
\end{align*}
This proves the claimed pointwise convergence
$\exp(\phi^\epsilon_Y(u))\rightarrow \exp(-u^2/2)$ as $\epsilon\to 0$.

    Next we argue that condition~(\ref{eq:NecessaryAndSufficient}) is in fact sufficient for the process-level
    convergence claimed in the theorem. Note that the result of the claim
    also implies uniform convergence of the relevant characteristic functions
    on compact intervals~\cite[Theorem~5.3, p.~86]{ModernProbabilityFoundations}. Now, it is straightforward, from the definition of a Lévy process, that the increments \(Y_\epsilon(t)-Y_\epsilon(s)\) converge in distribution to the corresponding (Gaussian) increments of a Brownian motion, for all $\ s < t$. We proceed to verify requirement~(III) 
     in~\cite[Theorem~V.19]{Pollard_1984}, i.e.
    {
     given $\delta > 0$, there are $\alpha  > 0$ 
and $\beta>0$
and $\epsilon_0 > 0$ such
that ${\mathbb P}\{|Y_\epsilon(t) - Y_\epsilon(s)| \leq \delta\} \geq \beta$
whenever $|t - s| < \alpha$ and $\epsilon < \epsilon_0$.
To see this, note that the L\'{e}vy process
$(Y_\epsilon(t) : t \in [0, 1])$ is centered, so for any 
$1 \geq t \geq s \geq 0$, by Chebychev,
\begin{align*}
{\mathbb P}\{|Y_\epsilon(t) - Y_\epsilon(s)| > \delta\} 
&={\mathbb P}\{|Y_\epsilon(t - s)| > \delta\}\\
&\leq \frac{1}{\delta^2} {\rm Var}(Y_\epsilon(t - s))\\
&=\frac{(t - s){\rm Var}(Y_\epsilon(1))}{\delta^2}= \frac{(t - s)}{\delta^2}.
\end{align*}
By taking $\delta\in (0, 1)$, we can choose any
$\epsilon_0$ and $\alpha = \delta^3$. Then, ${\mathbb P}\{|Y_\epsilon(t) - Y_\epsilon(s)| > \delta\}=1-{\mathbb P}\{|Y_\epsilon(t) - Y_\epsilon(s)| \leq \delta\}$, so we can set  $\beta=1-\delta$,  
and we conclude \((Y_{\epsilon}(t)) \xrightarrow{\epsilon \rightarrow 0}{} (W(t))\) in $D[0,1]$, as claimed.}
    \end{proof}
Next we show that some conditions are in fact necessary for the Gaussian limit in Theorem~\ref{theorem:NVMConvergenceGaussian} to hold.

\begin{theorem}
\label{theorem:NVMnecessity}
Consider a truncated NVM L\'{e}vy process $X_{\epsilon}=(X_{\epsilon}(t))$ and define
the associated standardised process $Y_\epsilon$ as in Theorem~\ref{theorem:NVMConvergenceGaussian}.
If condition~{\em (\ref{eq:NecessaryAndSufficient})} does not hold and, moreover,
\begin{align}
L_1:=\liminf_{\epsilon \rightarrow 0}\frac{M_{Z_\epsilon}^{(2)}}{{M_{Z_\epsilon}^{(1)}}^2} >0
\quad\mbox{and}\quad
L_2:=\sigma_W^6\limsup_{\epsilon\rightarrow 0}\frac{M^{(3)}_{Z_\epsilon}}{\sigma_\epsilon^6}>0.
\label{eq:necessity}
\end{align}
then $Y_\epsilon(1)$ does not converge to ${\cal N}(0,1)$ in distribution as $\epsilon\to 0$. 
\end{theorem}

\begin{proof}
Suppose the conditions in~(\ref{eq:necessity}) hold. We will assume that
$Y_\epsilon(1)$ converges to ${\cal N}(0,1)$ in distribution as $\epsilon\to 0$,
and derive a contradiction.

In the notation of the proof of Theorem~\ref{theorem:NVMConvergenceGaussian},
expanding the exponential 
series in~(\ref{eq:shorts}) for one more term than before and using~(\ref{eq:BoundExponential}), yields,
\begin{align*}
    \int_0^\epsilon(\exp(vz/\sigma_\epsilon^2)-1-vz/\sigma_\epsilon^2-v^2z^2/(2\sigma_\epsilon^4))Q(dz)=D_\epsilon(v),
    \end{align*}
with    $|D_\epsilon(v)|\leq |v|^3M^{(3)}_{Z_\epsilon}/(3!\sigma_\epsilon^6)$.
Hence, rearranging and integrating:
\begin{equation}
\varphi_\epsilon(u)=\int_0^\epsilon(\exp(vz/\sigma_\epsilon^2)-1-vz/\sigma_\epsilon^2)Q(dz)=F_\epsilon(v)+D_\epsilon(v),
    \label{eq:longs}
\end{equation}
where $v={iu\mu_{W}}{\sigma_\epsilon}-{u^2\sigma_{W}^2}/2$ as before and
$F_\epsilon(v):=v^2 M^{(2)}_{Z_\epsilon}/(2\sigma_\epsilon^4).$
The assumption that $Y_\epsilon(1)$ converges to ${\cal N}(0,1)$ implies
that $|\varphi_\epsilon(u)|\to 0$ for all $u$.

We consider several cases. 
First, we note that $F_\epsilon(v)$ cannot converge to zero
for any $u\neq 0$, since, 
by~(\ref{eq:necessity}):
\begin{align}
    \liminf_{\epsilon\rightarrow 0} {|}F_\epsilon(v){|}
    =\frac{1}{4}u^4\sigma_{W}^4\liminf_{\epsilon\rightarrow 0}\frac{M^{(2)}_{Z_\epsilon}}{\sigma_\epsilon^4}
    =\frac{1}{4}u^4\sigma_{W}^4\liminf_{\epsilon \rightarrow 0}\frac{M^{(2)}_{Z_\epsilon}}{{M^{(1)}_{Z_\epsilon}}^2}
    =\frac{1}{4}u^4\sigma_{W}^4 L_1>0.
    \label{eq:firstterm}
\end{align}
Second, we note that $D_\epsilon(v)$ also cannot converge to zero, because
then $F_\epsilon(v)$ would need to converge to zero as well.
The only remaining case is if, for every $u$, neither $F_\epsilon(v)$ nor $D_\epsilon(v)$ vanish as $\epsilon\to0$,
while their sum does vanish. By~(\ref{eq:necessity}) we have,
\[
\limsup_{\epsilon\to 0} |D_\epsilon(v)|\leq  \limsup_{\epsilon\to 0}|v|^3\frac{M^{(3)}_{Z_\epsilon}}{3!\sigma_\epsilon^6}
={\frac{1}{48}u^6L_2}.
\]
On the other hand, by~(\ref{eq:longs}) and~(\ref{eq:firstterm}) we have,
\[
\liminf_{\epsilon\to 0} |D_\epsilon(v)|
=
\liminf_{\epsilon\to 0} \Big| |\varphi_\epsilon(v)| - F_\epsilon(v)\Big|
=
\liminf_{\epsilon\to 0}  |F_\epsilon(v)|=
\frac{1}{4}u^4\sigma_{W}^4 L_1.
\]
Therefore, 
$${\frac{1}{48}u^6L_2}\geq\limsup_{\epsilon\to 0} |D_\epsilon(v)|\geq \liminf_{\epsilon\to 0} |D_\epsilon(v)|=\frac{1}{4}\sigma_{W}^4 L_1u^4.$$
Since $L_1,L_2>0$, this is clearly violated for $u$ small enough, 
providing the desired contradiction and completing the proof.
\end{proof}

\begin{corollary}
\label{corollary:SufficientCondition}
Suppose the subordinator Lévy measure \(Q_{Z}(dz)\) admits a density \(Q_{Z}(z)\), which, for some $E>0$, satisfies $Q_{Z}(z) > 0$  for all $z \in (0, E]$. Then condition~{\em (\ref{eq:NecessaryAndSufficient})} in Theorem~\ref{theorem:NVMConvergenceGaussian} is equivalent to:
$$
    \lim_{\epsilon\rightarrow0}\epsilon Q_{Z}(\epsilon) = +\infty.
$$
\end{corollary}

\begin{proof}
A straightforward application of L'Hôpital's rule twice gives,
\begin{align*}
    \lim_{\epsilon\rightarrow0} \frac{M^{(2)}_{Z_\epsilon}}{M^{(1)^{2}}_{Z_\epsilon}} = \lim_{\epsilon\rightarrow0}\frac{\int_{0}^{\epsilon}z^{2}Q_{Z}(z)dz}{\left(\int_{0}^{\epsilon}zQ_{Z}(z)dz\right)^{2}} = \lim_{\epsilon\rightarrow0} \frac{\epsilon}{2M^{(1)}_{Z_\epsilon}} = \lim_{\epsilon\rightarrow0}\frac{1}{2\epsilon Q_{Z}(\epsilon)},
\end{align*}
where the positivity of $Q_Z(z)$ in $(0, E]$ was required to ensure that the derivative of the denominator is non-zero for all $\epsilon$
small enough.
\end{proof}

\section{Bounds on the marginal convergence rate}

In addition to the asymptotic convergence of the process to a Brownian motion as in Theorems \ref{theorem:NVMConvergenceGaussian} and \ref{theorem:NVMnecessity}, it is possible to compute finite-$\epsilon$ bounds on the distance from Gaussian of the marginals. The following theorem, based on Berry-Ess\'{e}en-style arguments, gives a general result. It may be observed once again (see Corollary~\ref{corollary:SufficientCondition}) that the quantity $\epsilon/M^{(1)}_{Z_\epsilon}$ is of importance in determining the performance of different subordinators $Q_Z$. In the following section we study specific NVM processes within this framework.

\begin{theorem}\label{theorem:RoC}
Consider a truncated NVM L\'{e}vy process $X_{\epsilon}=(X_{\epsilon}(t))$ and let the standardised process $Y_\epsilon$ be defined as 
in Theorem~\ref{theorem:NVMConvergenceGaussian}. Then the Kolmogorov distance $E_\epsilon$ between \(Y_{\epsilon}(1)\) and \(B\sim{\cal N}(0,1)\) satisfies,
        \begin{align}
E_\epsilon
&:= 
    \sup_{x \in \mathbbm{R}}\left|\mathbbm{P}\left[Y^\epsilon(1) \leq x\right]-\mathbbm{P}\left[B \leq x\right]\right|
    \nonumber\\
&\leq
    C   \sigma_{W}^{3}\Phi\left(-\frac{3}{2}, \frac{1}{2}; -\frac{\mu_{W}^{2}}{2\sigma_{W}^{2}}\epsilon\right)\frac{M^{(\frac{3}{2})}_{Z_\epsilon}}{\sigma_{\epsilon}^{3}} \label{eq:NVM_RateOfConvergence}\\
&=C   \sigma_{W}^{3}\frac{M^{(\frac{3}{2})}_{Z_\epsilon}}{\sigma_{\epsilon}^{3}}\big(1+{\cal O}\left(\epsilon\right)\big),
    \qquad\mbox{as}\;\epsilon\to 0,\label{eq:NVM_asympt_rateOfConvergence}
        \end{align}
    where $C=0.7975 \times 2\sqrt{2/\pi}$,
    $\Phi(a, b; m)$ is  the Kummer confluent hypergeometric
    function, and with the obvious extension of~{\em (\ref{eq:Z_moments})} to non-integer moments.
    
    Furthermore, with the same constant $C$, $E_\epsilon$ may be bounded, for $\epsilon\in(0,1]$, as,
    \begin{align} 
    E_\epsilon 
    &\leq C\Phi\left(-\frac{3}{2}, \frac{1}{2}; -\frac{\mu_{W}^{2}}{2\sigma_{W}^{2}}\epsilon\right)\Bigg(\frac{\epsilon}{M^{(1)}_{Z_\epsilon}}\Bigg)^{1/2}
    \label{presimplified_RoC}\\
    &=C\Bigg(\frac{\epsilon}{M^{(1)}_{Z_\epsilon}}\Bigg)^{1/2}\big(1+{\cal O}\left(\epsilon\right)\big),
    \qquad\mbox{as}\;\epsilon\to 0.
    \label{simplified_RoC}
    \end{align}
    \end{theorem}
    \begin{proof}
        Arguing as in the proof of~\cite[Theorem~3.1]{LevyStateSpaceModel}, which was derived from~\cite[Theorem~2.1]{SmallJumps}, the Kolmogorov distance $E_\epsilon$ between \(Y^\epsilon(1)\) and \(B\sim{\cal N}(0,1)\) is bounded above by,
        \begin{align}
           E_\epsilon= \sup_{x \in \mathbbm{R}}\left|\mathbbm{P}\left[Y^\epsilon(1) \leq x\right]-\mathbbm{P}\left[B \leq x\right]\right| \leq 0.7975\sigma_{\epsilon}^{-3}\int_{\mathbbm{R}}|x|^{3}Q^{\epsilon}_{X}(dx), \label{eq:RateOfConvergence}
        \end{align}
        where \(\sigma_{\epsilon}^{2}\) is the variance of the NVM process. From~(\ref{eq:NVMVarianceSub}) it follows that,
        \begin{equation}
            \sigma_{\epsilon}^{2} = \mu_{W}^{2}M_{Z_{\epsilon}}^{(2)} + \sigma_{W}^{2}M_{Z_{\epsilon}}^{(1)}
            \geq  \sigma_{W}^{2}M_{Z_{\epsilon}}^{(1)}.\label{eq:sig_sq_bound}
        \end{equation}
        Using Fubini's theorem, the third absolute moment of the residual process can be expressed as,
        \begin{align*}
            \mathcal{S} &:= \int_{\mathbbm{R}}|x|^{3}Q^{\epsilon}_{X}(dx)\\
            &= \int_{-\infty}^{\infty}|x|^{3}\int_{0}^{\epsilon}\mathcal{N}(x;\mu_{W}z, \sigma_{W}^{2}z)Q_{Z}(dz)dx\\
            & = \int_{0}^{\epsilon}z^{\frac{3}{2}}\sigma_{W}^{3}2^{\frac{3}{2}}\frac{\Gamma(2)}{\sqrt{\pi}}\Phi\left(-\frac{3}{2}, \frac{1}{2}; -\frac{\mu_{W}^{2}}{2\sigma_{W}^{2}}z\right)Q_{Z}(dz),
        \end{align*}
    and since the Kummer confluent hyper-geometric function is increasing for non-negative \(z\)~\cite{KummerSeries}, we can bound ${\cal S}$ as, 
\begin{align*}
            \mathcal{S}& \leq
            \sigma_{W}^{3}2^{\frac{3}{2}}\frac{\Gamma(2)}{\sqrt{\pi}}\Phi\left(-\frac{3}{2}, \frac{1}{2}; -\frac{\mu_{W}^{2}}{2\sigma_{W}^{2}}\epsilon\right)\int_{0}^{\epsilon}z^{\frac{3}{2}}Q_{Z}(dz)\\
            &= \sigma_{W}^{3}2^{\frac{3}{2}}\frac{\Gamma(2)}{\sqrt{\pi}}\Phi\left(-\frac{3}{2}, \frac{1}{2}; -\frac{\mu_{W}^{2}}{2\sigma_{W}^{2}}\epsilon\right)M^{(\frac{3}{2})}_{Z_\epsilon}.
        \end{align*}
Substituting this bound into~(\ref{eq:RateOfConvergence}) and using (\ref{eq:sig_sq_bound}), we obtain (\ref{eq:NVM_RateOfConvergence}). Then, using the expansion,
  \begin{equation}
        \Phi(a,b;z)=\sum_{n=0}^\infty \frac {a^{(n)} z^n} {b^{(n)} n!},\label{series_hyper}
        \end{equation}
        where
        $a^{(0)}=1$ and $a^{(n)}=a(a+1)(a+2)\cdots(a+n-1)$, we obtain,
        \[
        \Phi\left(-\frac{3}{2}, \frac{1}{2}; -\frac{\mu_{W}^{2}}{2\sigma_{W}^{2}}\epsilon\right)=1+\frac{3\mu_{W}^{2}}{2\sigma_{W}^{2}}\epsilon+{\cal O}(\epsilon^2),
        \]
from which the asymptotic expansion (\ref{eq:NVM_asympt_rateOfConvergence}) is obtained.

Now, the term $M^{(\frac{3}{2})}_{Z_\epsilon}/\sigma_{\epsilon}^{3}$ is bounded using (\ref{M_condition}) and (\ref{eq:sig_sq_bound}), for $\epsilon\in(0,1]$, as,
\[
\frac{M^{(\frac{3}{2})}_{Z_\epsilon}}{\sigma_{\epsilon}^{3}}\leq \frac{\epsilon^{1/2}M^{(1)}_{Z_\epsilon}}{\sigma_{W}^{3}{M^{(1)}_{Z_\epsilon}}^{3/2}}=\frac{\epsilon^{1/2}}{\sigma_{W}^{3}{M^{(1)}_{Z_\epsilon}}^{1/2}},
\]
which yields (\ref{presimplified_RoC}), and applying (\ref{series_hyper}) once again leads to (\ref{simplified_RoC}).
\end{proof}

\section{Examples}\label{section:ExampleCases}
In this section, the validity of the conditions in Theorems~\ref{theorem:NVMConvergenceGaussian}
and~\ref{theorem:NVMnecessity} on the Gaussian convergence of the residual process is examined for several important cases of NVM L\'{e}vy processes. Simulation results validating the corresponding conclusions are also shown, and explicit bounds on the rate of convergence to the Gaussian are derived in some special cases, using the general framework of Theorem \ref{theorem:RoC}.

\subsection{Normal-Gamma (NG) process \label{section:NGProof}}
    The subordinator of the NG process is a Gamma process, with parameters $\nu, \gamma>0$ and with Lévy density:
    \begin{equation*}
        Q_{Z}(z) = \nu z^{-1} \exp\left(-\frac{1}{2}\gamma^{2}z\right), \quad z>0.
    \end{equation*}
Here, in view of Corollary~\ref{corollary:SufficientCondition},
\begin{align*}
    \lim_{\epsilon\rightarrow0}\frac{M_{Z_\epsilon}^{(2)}}{M_{Z_\epsilon}^{(1)^2}}=
    \lim_{\epsilon\rightarrow0}\frac{1}{2\epsilon Q_{Z}(\epsilon)} = \frac{1}{2\nu} > 0,
\end{align*}
and also,
\begin{align*}
    \lim_{\epsilon\to0}\frac{M_{Z_\epsilon}^{(3)}}{\sigma_\epsilon^6}= \frac{1}{3\nu^{2}\sigma_{W}^{6}} >0.
\end{align*}
Therefore, $L_1$ and $L_2$ in Theorem~\ref{theorem:NVMnecessity} are both nonzero,
so we expect the residuals of the NG process  {\em not} to be approximately normally distributed.

Furthermore, since $M^{(n)}_{Z_\epsilon}=\frac{\nu}{b^n}\gamma(n,b\epsilon)$, where $b=\gamma^{2}/2$, we have from (\ref{eq:NVM_RateOfConvergence}) that,
 \begin{align*}
           E_\epsilon \leq C   \sigma_{W}^{3}\Phi\left(-\frac{3}{2}, \frac{1}{2}; -\frac{\mu_{W}^{2}}{2\sigma_{W}^{2}}\epsilon\right)\frac{\frac{\nu}{b^{3/2}}\gamma(3/2,b\epsilon)}{(\mu_W^2\frac{\nu}{b^2}\gamma(2,b\epsilon)+\sigma_W^2\frac{\nu}{b}\gamma(1,b\epsilon))^{3/2}}={\cal O}(1),\,\,\text{{\em as\/} $\epsilon\to 0$},
        \end{align*}
where we have obtained the asymptotic behaviour using $\gamma(s, x) =  x^s\,\sum_{k=0}^\infty \frac{(-x)^k}{k!(s+k)}$ (for positive, real $s$ and $x$) and (\ref{series_hyper}). Hence, as expected, the bound on the distance from Gaussianity, $E_\epsilon$, does not tend to zero as $\epsilon\to 0$. 

In order to verify this empirically, we generate a random sample of \(M=10^4\) residual NG paths
with parameters $\mu=0$, $\mu_w=1$, $\sigma_W=\nu=2$ and $\gamma=\sqrt{2}$, and compare the empirical distribution of the values of the residual at time \(t=1\) with a standard normal distribution by standardising the residual values to have zero mean and unit variance. As expected, the resulting histograms are not approximately Gaussian.
Figures~\ref{fig:NGCLTVerification1} and~\ref{fig:NGCLTVerification2} show the distribution of the residual (with truncation level $\epsilon=10^{-6}$)
is in fact leptokurtic and heavier-tailed than the standard normal. Further simulations confirmed this empirical observation even for smaller truncation levels $\epsilon$. 
\begin{figure}[H]
    \begin{centering}
    \scalebox{0.55}{\includegraphics{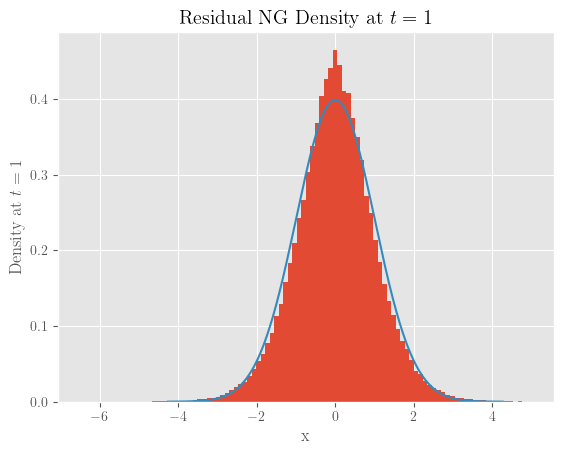}}
    \caption{Histogram of \(M=10^4\) NG residual path values at \(t=1\). The blue curve represents the standard normal density.} 
    \label{fig:NGCLTVerification1}
    \end{centering}
\end{figure}
\begin{figure}[H]
    \begin{centering}  
    \scalebox{0.55}{\includegraphics{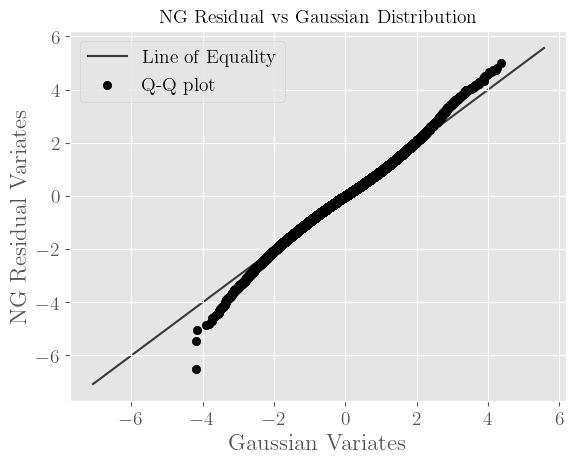}}
    \caption{Q-Q plot of \(M=10^5\) NG residual path values at \(t=1\).}
    \label{fig:NGCLTVerification2}
    \end{centering}
\end{figure}
    
\subsection{Normal tempered stable (NTS) process \label{section:NTSProof}}
    The subordinator for the NTS process is the tempered stable (TS) process TS\((\kappa, \delta, \gamma)\),  for \( \kappa \in (0,1), \ \delta > 0, \ \gamma \geq 0\), which has a Lévy density,
    \begin{equation*}
        Q_{Z}(z)  =  Az^{-1-\kappa}\exp\left(-\frac{1}{2}\gamma^{\frac{1}{\kappa}}z\right), \quad z > 0,
    \end{equation*}
    where \(A = \delta\kappa 2^{\kappa}\Gamma^{-1}(1-\kappa)\) and \(\Gamma^{-1}(\cdot)\) is the reciprocal of the Gamma function. 
    Here, 
    \begin{equation*}
        \lim_{\epsilon\rightarrow 0}\epsilon Q_{Z}(\epsilon) = \lim_{\epsilon\rightarrow 0}A\epsilon^{-\kappa}\exp\left(-\frac{1}{2}\gamma^{\frac{1}{\kappa}}\epsilon\right) =  +\infty, 
    \end{equation*}
    since $\kappa\in(0,1)$. Therefore, in view of Corollary~\ref{corollary:SufficientCondition} and 
    Theorem~\ref{theorem:NVMConvergenceGaussian}, the residuals are expected to be approximately
    Gaussian. Moreover, in this case we can derive a bound on the corresponding marginal convergence rate. 
    

    \begin{lemma}\label{lemma:NTSRoC}
    For $\epsilon\in(0,1)$, let $(Y_{\epsilon}(t))$ denote the standardised truncated process associated to an NVM process subordinated to the residual TS process ${\rm TS}(\kappa, \delta, \gamma)$. If \(B\sim\mathcal{N}(0,1)\), then the Kolmogorov distance $E_\epsilon$ between \(Y_{\epsilon}(1)\) and \(B\) satisfies,
    \begin{align}
        \label{eq:NTSbound}
         E_{\epsilon} \leq 
            \frac{0.7975\times 2^{\frac{3}{2}}\sqrt{\Gamma(1-\kappa)}}{\sqrt{\delta\kappa\pi\gamma}}\times
            \Phi_{\epsilon}\gamma\left(1-\kappa,\frac{1}{2}\epsilon\gamma^{\frac{1}{\kappa}}\right)^{-\frac{3}{2}}
            \gamma\left(\frac{3}{2}-\kappa, \frac{1}{2}\gamma^{\frac{1}{\kappa}}\epsilon\right),
    \end{align}
    where,
    \begin{align*}
        \Phi_{\epsilon} = \Phi\left(-\frac{3}{2}, \frac{1}{2}; -\frac{\mu_{W}^{2}}{2\sigma_{W}^{2}}\epsilon\right),
    \end{align*}
    and \(\Phi(a, b; m), \gamma(s,x)\) are the Kummer confluent hypergeometric function and the incomplete lower gamma function, respectively. Further, as \(\epsilon\rightarrow 0\) we have:
    \begin{equation}
        E_{\epsilon} \leq \frac{0.7975\times 2^{\frac{3}{2}-\frac{\kappa}{2}}(1-\kappa)^\frac{3}{2}\sqrt{\Gamma(1-\kappa)}}{(\frac{3}{2}-\kappa)\sqrt{\delta\kappa\pi}}\epsilon^{\frac{\kappa}{2}}+\mathcal{O}(\epsilon^{1+\frac{\kappa}{2}}).
        \label{eq:lemasymptotic}
    \end{equation}
    \end{lemma}
    \begin{proof}
        From Theorem \ref{theorem:RoC} we obtain,
        $E_{\epsilon} \leq C  \sigma_{W}^{3}\Phi_{\epsilon}M^{(\frac{3}{2})}_{Z_\epsilon}/\sigma_{\epsilon}^{3}$,
        and from (\ref{eq:NVMVarianceSub}) it follows that,
        \begin{equation*}
            \sigma_{\epsilon}^{2} = \mu_{W}^{2}M_{Z_{\epsilon}}^{(2)} + \sigma_{W}^{2}M_{Z_{\epsilon}}^{(1)}
            \geq  \sigma_{W}^{2}M_{Z_{\epsilon}}^{(1)},
        \end{equation*}
        where the residual first moment is given by,
        $$
            M^{(1)}_{Z_{\epsilon}} = \int_{0}^{\epsilon}zAz^{-1-\kappa}\exp\left(-\frac{1}{2}\gamma^{\frac{1}{\kappa}}z\right)dz=A \gamma^{\frac{\kappa-1}{\kappa}}2^{1-\kappa}\gamma\left(1-\kappa,\frac{1}{2}\epsilon\gamma^{\frac{1}{\kappa}}\right). $$
        To find \(M_{Z_{\epsilon}}^{\left(\frac{3}{2}\right)}\), substitute the Lévy density of the TS process described earlier,
        \begin{align*}
            M_{Z_{\epsilon}}^{\left(\frac{3}{2}\right)} &=  A\int_{0}^{\epsilon}z^{\frac{1}{2}-\kappa}e^{-\frac{1}{2}\gamma^{\frac{1}{\kappa}}z}dz =A\left(\frac{1}{2}\gamma^{\frac{1}{\kappa}}\right)^{\kappa - \frac{3}{2}}\gamma\left(\frac{3}{2}-\kappa, \frac{1}{2}\gamma^{\frac{1}{\kappa}}\epsilon\right),
        \end{align*}
        Substituting this in~(\ref{eq:RateOfConvergence}), noting \(A = \delta \kappa 2^{\kappa}\Gamma^{-1}(1-\kappa)\) and \(C = 0.7975 \times 2\sqrt{2/\pi}\) yields, 
        \begin{align*}
        E_{\epsilon} \leq 
        0.7975\times \frac{2^{\frac{3}{2}}\sqrt{\Gamma(1-\kappa)}}{\sqrt{\delta\kappa\pi\gamma}}\Phi_{\epsilon}\gamma\left(1-\kappa,\frac{1}{2}\epsilon\gamma^{\frac{1}{\kappa}}\right)^{-\frac{3}{2}}\gamma\left(\frac{3}{2}-\kappa, \frac{1}{2}\gamma^{\frac{1}{\kappa}}\epsilon\right),
        \end{align*}
        as claimed. Finally, recalling the series expansion for $\gamma(s,x)$ from the previous section, the asymptotic expression in (\ref{eq:NVM_asympt_rateOfConvergence}) leads to (\ref{eq:lemasymptotic}).
    \end{proof}
    Next, we examine the empirical accuracy of the Gaussian approximation in this case.
    Figures~\ref{fig:NTSCLTVerification1} and~\ref{fig:NTSCLTVerification2} show the empirical distribution 
    of the residual NTS process at \(t=1\), with parameter values $\mu=0$, $\mu_W=\delta=1$, $\sigma_W=2$, $\kappa=1/2$,
    $\gamma=1.35$,
    and truncation level $\epsilon=10^{-6}$.

    \begin{figure}[H]
        \begin{centering}
        \scalebox{0.55}{\includegraphics{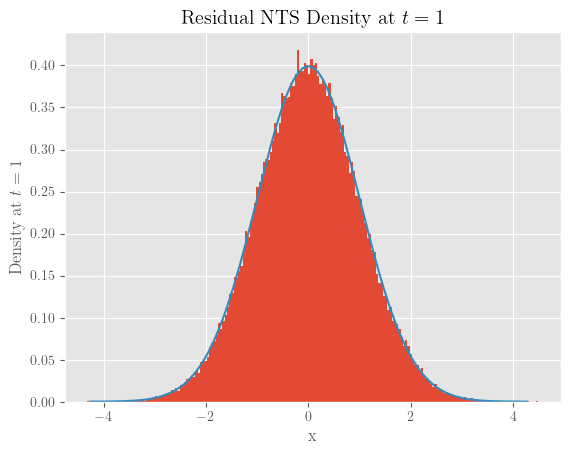}}
        \caption{Histogram of \(M = 10^5\) NTS residual path values at \(t=1\). The blue curve represents the standard normal density. \label{fig:NTSCLTVerification1}}
        \end{centering}
    \end{figure}
    
    \begin{figure}[H]
        \begin{centering}
        \scalebox{0.55}{\includegraphics{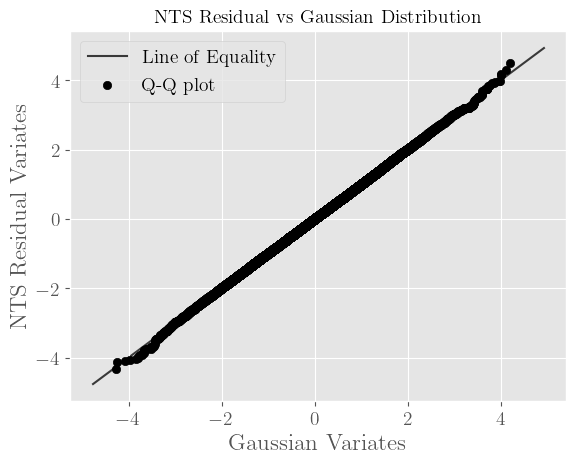}}
        \caption{Q-Q plot of \(M = 10^5\) NTS residual path values at \(t=1\).  \label{fig:NTSCLTVerification2}}
        \end{centering}
    \end{figure}

With the same parameter values, Figure~\ref{fig:NTSConvergenceRate} shows behaviour of the bound in~(\ref{eq:NTSbound})
and the first term in the asymptotic bound~(\ref{eq:lemasymptotic}).

     \begin{figure}[H]
        \begin{centering}
        \scalebox{0.55}{\includegraphics{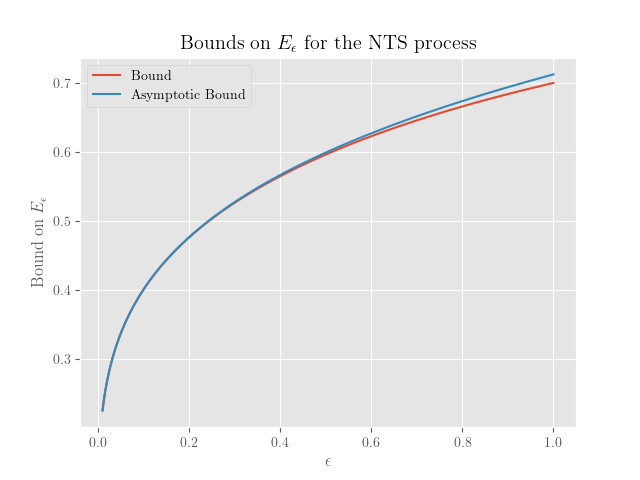}}
        \caption{Plot of the finite-$\epsilon$ bound in~(\ref{eq:NTSbound})
        and the first term in the asymptotic bound~(\ref{eq:lemasymptotic}) 
        for the approximation error $E_\epsilon$ in Lemma~\ref{lemma:NTSRoC}.}
        \label{fig:NTSConvergenceRate}
        \end{centering}
    \end{figure}
    
\subsection{Generalised hyperbolic (GH) process \label{section:GHProof}}
    Finally, we consider the general class of GH processes. The subordinator in this case is
    the generalized inverse Gaussian
    GIG(\(\lambda, \delta, \gamma\)) process, with constraints on parameter values as detailed in~\cite{GIGLevy}, 
    and with Lévy density given by~\cite{GIGLevy},
    \begin{equation*}
        Q_{Z}(z) = \frac{\exp \left(-z\gamma^{2}/2\right)}{z}
        \left[\max(0, \lambda)+\frac{2}{\pi^{2}}\int_{0}^{\infty}\frac{1}{y\big|H_{|\lambda|}(y)\big|^{2}} 
        \exp \Big( -\frac{zy^{2}}{2\delta^{2}}\Big) dy\right]\mathbbm{1}(z > 0),
    \end{equation*}
    where \(H_{v}(z)\) is the Hankel function of real order \(v\). { A direct verification of  the sufficient condition in Corollary \ref{corollary:SufficientCondition} is readily obtained as  
    
    \begin{align*}
        \lim_{\epsilon\rightarrow0}\epsilon Q_{Z}(\epsilon) 
        &= \lim_{\epsilon\rightarrow 0}\exp\Big(-\frac{\gamma^{2}}{2}\epsilon\Big)\max(0,\lambda) 
            + \frac{2}{\pi^{2}}\lim_{\epsilon\rightarrow0}\int_{0}^{\infty}\frac{1}{y\big|H_{|\lambda|}(y)\big|^{2}} 
            \exp\Big(-\frac{\epsilon\gamma^{2}}{2} -\frac{\epsilon y^{2}}{2\delta^{2}}\Big) dy \\
        &= \max(0,\lambda) + \frac{2}{\pi^{2}}\int_{0}^{\infty}\frac{1}{y\big|H_{|\lambda|}(y)\big|^{2}}dy\\
        &=  +\infty,
    \end{align*}
    since \({y\big|H_{\nu}(y)\big|^{2}}\) is non-zero for $z\in [z_1,\infty)$ where $z_1 = \left(\frac{ 2^{1-2\nu}\pi}{\Gamma^2(\nu)}\right)^{1/(1-2\nu)}$, see e.g. \cite{GIGLevy} Theorem 2 and \cite{Watson1944}.}

    Once again, in view of Corollary~\ref{corollary:SufficientCondition} and 
    Theorem~\ref{theorem:NVMConvergenceGaussian}, the residuals are expected to be approximately
    Gaussian. Indeed, in this case we can derive the following bound on the corresponding marginal convergence rate.
    
    \begin{lemma}\label{lemma:GHRoC}
    
    For $\epsilon\in(0,1)$, let \(Y_{\epsilon}(t)\) denote denote the standardised truncated process associated to
    an NVM process subordinated to the residual \({\rm GIG (}\lambda, \delta, \gamma{\rm )}\) process. If \(B\sim\mathcal{N}(0,1)\), then
    for any $z_0\in(0,\infty)$ the Kolmogorov distance $E_\epsilon$ between \(Y_{\epsilon}(1)\) and \(B\) can be bounded as,
    \begin{align}
    E_{\epsilon}
        &\leq \frac{{0.7975\Phi_\epsilon\gamma^{3/2}}}{{\rm erf}\left(\gamma\frac{\sqrt{\epsilon}}{\sqrt{2}}\right)^{\frac{3}{2}}} \biggl(
        \frac{2\max(0,\lambda)}{{\Tilde{\pi}}(b\delta)^{3/2}}\gamma\left(\frac{3}{2}, b\epsilon\right)
        + \frac{2^{|\lambda|{+1}}\delta^{2|\lambda|-\frac{3}{2}}\Gamma(|\lambda|)}
        {{\pi^{2}\Tilde{\pi}} H_{0}z_{0}^{2|\lambda|-1}b^{3/2-|\lambda|}}\times
        \gamma\left(\frac{3}{2}-|\lambda|, b\epsilon\right)
            \nonumber\\
        &\qquad \qquad\qquad\qquad \qquad\qquad
        + \frac{1}{{\Tilde{\pi}^{4}}H_{0}b\sqrt{\delta}}\gamma\left(1, b\epsilon\right)\biggr),\qquad\mbox{for} \;|\lambda| \leq \frac{1}{2},
       \nonumber\\
    E_{\epsilon}
        & \leq \frac{{0.7975\Phi_\epsilon\left(\gamma H_0\right)^{3/2}}}{{\rm erfc}\left( \frac{z_{0}}{\delta\sqrt{2}}\sqrt{\epsilon}\right)^{\frac{3}{2}}{\rm erf}(\gamma\frac{\sqrt{\epsilon}}{\sqrt{2}})^{\frac{3}{2}}} \biggl( 
        {\frac{\max(0,\lambda)\pi}{(b\delta)^{\frac{3}{2}}}}
        \times
        \gamma\left(\frac{3}{2}, b\epsilon\right)
        \nonumber\\
        &\qquad \qquad\qquad\qquad \qquad\qquad
        + {\frac{\Tilde{\pi}}{b\sqrt{\delta}}}
        \times \gamma\left(1, b\epsilon\right)\biggr), \qquad\mbox{for}\;|\lambda| > \frac{1}{2},
        \label{eq:GHlemmabound}
    \end{align}
    where $b = \frac{\gamma^{2}}{2}$, $\Tilde{\pi} = \sqrt{\frac{\pi}{2}}$, $H_{0} = z_{0}\big|H_{|\lambda|}(z_{0})\big|^{2},$ and,
    \begin{align*}
        \Phi_{\epsilon} = \Phi\left(-\frac{3}{2}, \frac{1}{2}; -\frac{\mu_{W}^{2}}{2\sigma_{W}^{2}}\epsilon\right),
    \end{align*}
    with $\Phi(a, b; m)$, $\gamma(s,x)$, ${\rm erf}(x)$, and ${\rm erfc}(x)$\footnote{{Using the standard definitions ${\rm {erf}} (x) = \frac{2}{\sqrt\pi}\int_0^x e^{-t^2}\,\mathrm dt$ and ${\rm {erfc}} (x)=1-{\rm {erf}} (x)$}} denoting the Kummer confluent hypergeometric function, 
    the incomplete lower gamma function, the error function, and the complementary error function, respectively. Further, as \(\epsilon\rightarrow 0\) we have:
    \begin{align}
        E_{\epsilon} & \leq
            {\frac{0.7975}{\Tilde{\pi}^{\frac{5}{2}}bH_{0}\sqrt{\delta}}}\epsilon^{\frac{1}{4}}+\mathcal{O}(\epsilon^{\frac{5}{4}}), \qquad\mbox{for}\;|\lambda| \leq \frac{1}{2}, \nonumber \\
        E_{\epsilon} &\leq {\frac{0.7975 \Tilde{\pi}^{\frac{5}{2}}H_{0}^{\frac{3}{2}}}{b\sqrt{\delta}}}\epsilon^{\frac{1}{4}} + \mathcal{O}\left(\epsilon^{\frac{5}{4}}\right), \qquad\mbox{for}\;|\lambda| > \frac{1}{2}.\label{eq:lemasymptotic2}
    \end{align}
    \end{lemma}
    \begin{proof}
    { Recalling the definition of the Jaeger integral as
    \[J(z)=\int_0^\infty\frac{e^{-\frac{x^2z}{2\delta^2}}}{x|H_{|\lambda|}(x)|^2}dx\]
    we have from the 
     bounds obtained in Appendix~\ref{appendix:UpperLowerJaegerBounds},} for $|\lambda|\leq 1/2$,
    $$     J(z) 
         \geq 
            {\delta{\left(\frac{\pi}{2}\right)}^{3/2}z^{-\frac{1}{2}}},
    $$
    and for $|\lambda|>1/2$:
         $$
         J(z)
         \geq
         \frac{\delta^{2|\lambda|}2^{|\lambda|-1}}{H_{0}z_{0}^{2|\lambda|-1}}z^{-|\lambda|}\gamma\left(|\lambda|, \frac{z_{0}^{2}}{2\delta^{2}}z\right)+\frac{\delta}{H_{0}\sqrt{2}}z^{-\frac{1}{2}}\Gamma\left(\frac{1}{2},\frac{z_{0}^{2}}{2\delta^{2}}z\right).
         $$
    Noting that the variance of the GH process satisfies \(\sigma_{\epsilon}^{2} = \mu_{W}^{2}M^{(2)}_{Z_{\epsilon}} + \sigma_{W}^{2}M^{(1)}_{Z_{\epsilon}} \geq \sigma_{W}^{2}M^{(1)}_{Z_{\epsilon}}\), for \(|\lambda|\leq \frac{1}{2}\)
    we obtain,
    \begin{align*}
        M^{(1)}_{Z_{\epsilon}} 
        &= \int_{0}^{\epsilon}\exp \left(-\frac{\gamma^{2}z}{2}\right)\left[\max(0,\lambda) + \frac{2}{\pi^{2}}J(z)\right]dz\\
        &\geq \frac{2}{\pi^{2}}\int_{0}^{\epsilon}\exp \left(-\frac{\gamma^{2}z}{2}\right)J(z)dz\\
        &\geq {\frac{\delta}{\gamma}{\rm erf}\left(\frac{\gamma\sqrt{\epsilon}}{\sqrt{2}}\right)},
    \end{align*}
    and similarly for \(|\lambda|>\frac{1}{2}\) we obtain,
    \begin{align*}
        M^{(1)}_{Z_{\epsilon}}
          &\geq \frac{2}{\pi^{2}}\int_{0}^{\epsilon}e^{-\frac{z\gamma^{2}}{2}}\left[\frac{\delta^{2|\lambda|}2^{|\lambda|-1}}{H_{0}z_{0}^{2|\lambda|-1}}z^{-|\lambda|}\gamma\left(|\lambda|, \frac{z_{0}^{2}}{2\delta^{2}}z\right)+\frac{\delta}{H_{0}\sqrt{2}}z^{-\frac{1}{2}}\Gamma\left(\frac{1}{2},\frac{z_{0}^{2}}{2\delta^{2}}z\right)\right]dz\nonumber\\
          &\geq \frac{2}{\pi^{2}}\int_{0}^{\epsilon}e^{-\frac{z\gamma^{2}}{2}}\frac{\delta}{H_{0}\sqrt{2}}z^{-\frac{1}{2}}\Gamma\left(\frac{1}{2},\frac{z_{0}^{2}}{2\delta^{2}}z\right)dz\\
          &\geq \frac{2\delta}{H_{0}\pi\gamma}{\rm erfc}\left(\frac{z_{0}}{\delta\sqrt{2}}\sqrt{\epsilon}\right){\rm erf}\left(\frac{\gamma\sqrt{\epsilon}}{\sqrt{2}}\right).
    \end{align*}
    Therefore, we have the following bound on the variance:
    \begin{align*}
        \sigma_{\epsilon}^{2} \geq 
        \begin{cases}
            {\frac{\sigma_{W}^{2}\delta}{\gamma}}{\rm erf}\left(\frac{\gamma\sqrt{\epsilon}}{\sqrt{2}}\right),
            &\mbox{for}\;|\lambda| \leq \frac{1}{2}\\
            \frac{\sigma_{W}^{2}\delta}{H_{0}\Tilde{\pi}^{2}\gamma}{\rm erfc}\left( \frac{z_{0}}{\delta\sqrt{2}}\sqrt{\epsilon}\right)
            {\rm erf}\left(\frac{\gamma\sqrt{\epsilon}}{\sqrt{2}}\right),
            &\mbox{for}\;|\lambda| > \frac{1}{2}.
        \end{cases}
    \end{align*}
    Finding \(M_{Z_{\epsilon}}^{\left(\frac{3}{2}\right)}\) in line with Theorem \ref{theorem:RoC},
    \begin{align*}
        M_{Z_{\epsilon}}^{\left(\frac{3}{2}\right)}
        &= \int_{0}^{\epsilon}z^{\frac{3}{2}}Q_{Z}(dz)\\
        &=\max(0,\lambda) \int_{0}^{\epsilon}z^{\frac{1}{2}}\exp\left(-\frac{\gamma^{2}}{2}z\right)dz +\frac{2}{\pi^{2}}\int_{0}^{\epsilon}z^{\frac{1}{2}}\exp\left(-\frac{\gamma^{2}}{2}z\right)J(z)dz \\
        &= \mathcal{S}_{1} + \mathcal{S}_{2}.
    \end{align*}
    Writing \(b = \frac{\gamma^{2}}{2}\), we have,
    $$
        \mathcal{S}_{1} = \frac{\max(0,\lambda)}{b\sqrt{b}}\gamma\left(\frac{3}{2}, b\epsilon\right),
    $$
    for \(|\lambda| > \frac{1}{2}\) we have \({J(z)\leq \delta{\left(\frac{\pi}{2}\right)}^{3/2}z^{-\frac{1}{2}}}\) and
    $$
        \mathcal{S}_{2}\leq \frac{2}{\pi^{2}}\delta \left(\frac{\pi}{2}\right)^{\frac{3}{2}}\int_{0}^{\epsilon}\exp\left(-bz\right)dz =  \frac{\delta}{\sqrt{2\pi} b}\gamma\left(1, b\epsilon\right),
    $$
    and for \(|\lambda| \leq \frac{1}{2}\):
    \begin{align*}
        {\cal S}_{2} &\leq \frac{2^{|\lambda|}\delta^{2|\lambda|}}{\pi^{2}H_{0}z_{0}^{2|\lambda|-1}}\int_{0}^{\epsilon}z^{\frac{1}{2}-|\lambda|}e^{-bz}\gamma\left(|\lambda|, \frac{z_{0}^{2}}{2\delta^{2}}z\right)dz + \frac{\sqrt{2}\delta}{\pi^{2}H_{0}}\int_{0}^{\epsilon}e^{-bz}\Gamma\left(\frac{1}{2}, \frac{z_{0}^{2}}{2\delta^{2}}z\right)dz\\
        &\leq \frac{2^{|\lambda|}\delta^{2|\lambda|}\Gamma(|\lambda|)}{\pi^{2}H_{0}z_{0}^{2|\lambda|-1}}\int_{0}^{\epsilon}z^{\frac{1}{2}-|\lambda|}e^{-bz}dz + \frac{\delta}{\pi\tilde{\pi}H_{0}}\int_{0}^{\epsilon}e^{-bz}dz\\
        &\leq \frac{2^{|\lambda|}\delta^{2|\lambda|}\Gamma(|\lambda|)}{\pi^{2}H_{0}z_{0}^{2|\lambda|-1}b^{\frac{3}{2}-|\lambda|}}\gamma\left(\frac{3}{2}-|\lambda|, b\epsilon\right) + \frac{\delta}{\pi\tilde{\pi}H_{0}b}\gamma\left(1, b\epsilon\right).
    \end{align*}
    Combining the above bounds, for $|\lambda|\leq 1/2$,
    $$
    M_{Z_{\epsilon}}^{\left(\frac{3}{2}\right)} \leq \frac{\max(0,\lambda)}{b\sqrt{b}}\gamma\left(\frac{3}{2}, b\epsilon\right) + \frac{2^{|\lambda|}\delta^{2|\lambda|}\Gamma(|\lambda|)}{\pi^{2}H_{0}z_{0}^{2|\lambda|-1}b^{\frac{3}{2}-|\lambda|}}\gamma\left(\frac{3}{2}-|\lambda|, b\epsilon\right) + \frac{\delta}{\pi\tilde{\pi}H_{0}b}\gamma\left(1, b\epsilon\right),
    $$
    and for $|\lambda|>1/2,$
    $$
       M_{Z_{\epsilon}}^{\left(\frac{3}{2}\right)}\leq \frac{\max(0,\lambda)}{b\sqrt{b}}\gamma\left(\frac{3}{2}, b\epsilon\right)+ {\frac{\delta}{\sqrt{2\pi}b}}\gamma\left(1, b\epsilon\right).
    $$
    Finally, substituting these bounds for $M_{Z_{\epsilon}}^{\left(\frac{3}{2}\right)}$ into (\ref{eq:NVM_RateOfConvergence}) we obtain the bounds as stated in (\ref{eq:GHlemmabound}). The series expansions of the gamma and hypergeometric functions then lead to the asymptotic expansion (\ref{eq:lemasymptotic2}).
    \end{proof}

    Once again, we examine the validity of the Gaussian approximation empirically. 
    Figures~\ref{fig:GHCLTVerification1} and~\ref{fig:GHCLTVerification2}
    show the empirical distribution of the residual GH process at time $t = 1$, with
    parameter values $\mu = 0$, $\mu_W =1$, $\sigma_W = 2$, $\delta=1.3$, $\gamma=\sqrt{2}$, $\lambda=0.2$, and truncation level
    $\epsilon=10^{-6}$.
    
    \begin{figure}[H]
        \begin{centering}
        \scalebox{0.55}{\includegraphics{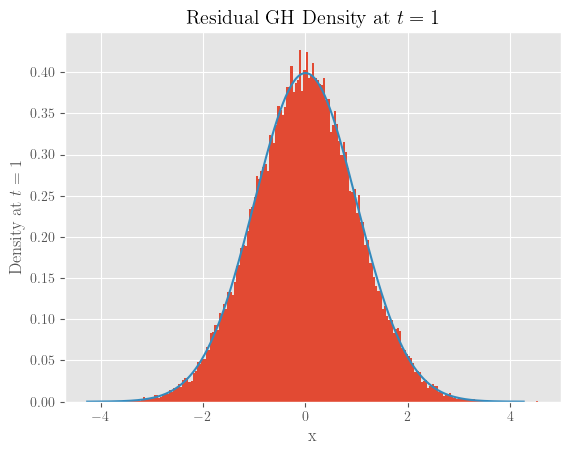}}
        \caption{Histogram of \(M = 5\times 10^4\) GH residual path values at \(t=1\). The blue curve represents the standard normal density.}
        \label{fig:GHCLTVerification1}
        \end{centering}	        
    \end{figure}
    \begin{figure}[H]
        \begin{centering}
        \scalebox{0.55}{\includegraphics{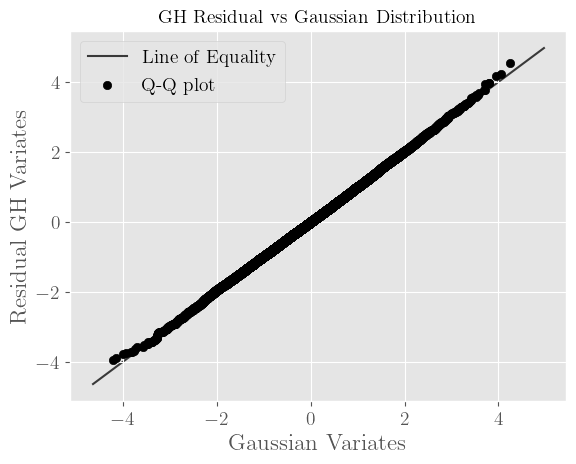}}
        \caption{Q-Q plot of \(M = 5\times 10^4\) GH residual path values at \(t=1\).}
        \label{fig:GHCLTVerification2}
		\end{centering}	        
    \end{figure}

    With the same parameter values, Figure~\ref{fig:GHConvergenceRate} shows behaviour of the bound in~(\ref{eq:GHlemmabound}) and the first term in the asymptotic bound~(\ref{eq:lemasymptotic2}).
     
     \begin{figure}[H]
        \begin{centering}
        \scalebox{0.55}{\includegraphics{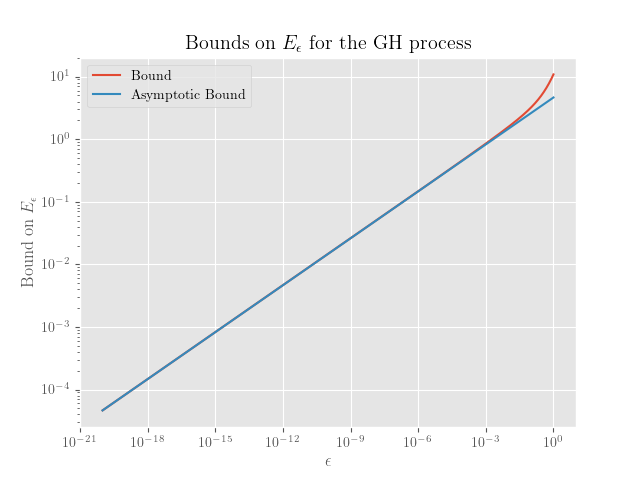}}
        \caption{Plot of the finite-$\epsilon$ bound in~(\ref{eq:GHlemmabound})
        and the first term in the asymptotic bound~(\ref{eq:lemasymptotic2}) 
        for the approximation error $E_\epsilon$ in Lemma~\ref{lemma:GHRoC}.} 
        \label{fig:GHConvergenceRate}
        \end{centering}
    \end{figure}
    
    Note that that bounds (\ref{eq:GHlemmabound}) and (\ref{eq:lemasymptotic2}) are discontinuous at \(|\lambda| = \frac{1}{2}\).
    This discrepancy is likely due to the upper bound for \(M_{Z_{\epsilon}}^{(1)}\) in the case \(|\lambda| > \frac{1}{2}\). 
    Although we do expect this could indeed be improved, obtaining such refined bounds is beyond the scope of this paper.  
    
\section{Linear SDEs}
The \textit{Lévy State Space Model}~\cite{LevyStateSpaceModel} defines a stochastic process having the following dynamics:
    $$
        d\boldsymbol{X}(t) = \boldsymbol{A}\boldsymbol{X}(t)dt + \boldsymbol{h}dW(t), \hspace{1cm} \boldsymbol{X}(t) \in \mathbbm{R}^{P}, W(t) \in \mathbbm{R}.%
    $$
where $\mbold{A}$ is a $P\times P$ matrix, $\mbold{h}\in\mathbbm{R}^P$. In~\cite{LevyStateSpaceModel} $(W(t))$ is assumed to follow a stable law; here it is taken to be an NVM L\'{e}vy process. The solution of the state process takes the form: 
    \begin{equation}
        \boldsymbol{X}(t) = e^{\mbold{A}(t-s)}\boldsymbol{X}(s) + \int_{s}^{t}e^{\mbold{A}(t-u)}\mbold{h}dW(u). \label{eq:SDE_solution}
    \end{equation}
We first present a shot-noise representation of the stochastic integral in (\ref{eq:SDE_solution}), and then prove the convergence of its small-jump residual to a Gaussian-driven SDE, under appropriate conditions. 

\subsection{Shot-noise representation of SDE \label{section:LevySSModel}}
    In order to apply the \textit{Lévy State Space Model} to NVM Lévy processes, we first establish their representation as generalised-shot noise series. Theorem \ref{theorem:StochIntegralSeries} gives the result for a general integrand:
    \begin{theorem}
    \label{theorem:StochIntegralSeries}
    Let $(X(u))$ be an NVM process generating the filtration \((\mathcal{F}_{t})\). Suppose
    \(\textit{\textbf{f}}_{t} : [0, \infty) \rightarrow \mathbbm{R}^{P}\) is an 
    \(L_{2}\) deterministic function, and let $\mbold{I}({\textbf{f}_{t}})$ denote the integral
    $$\mbold{I}({\textbf{f}_{t}}) =  \int_{0}^{T}\textit{\textbf{f}}_{t}(u)dX(u), \qquad 0\leq t\leq T.
    $$
    Then \(\mbold{I}({\textbf{f}_{t}})\) admits the series representation,
    \begin{equation*}
         \mbold{I}({\textbf{f}_{t}}) = \sum_{i=1}^{\infty}X_i\textit{\textbf{f}}_{t}(TV_{i})\mathbbm{1} \left( V_{i} \leq \frac{t}{T}\right),
    \end{equation*}
    where,
    \begin{align*}
    X_i=\mu_{W} Z_i+\sigma_{W}\sqrt{Z_i}U_{i},
    \end{align*}
  and \(\mu_{W} \in \mathbbm{R}, \ \sigma_{W} \in (0,\infty)\) are the variance-mean mixture parameters, \(V_{i} \overset{iid}{\sim} \mathcal{U}[0,1]\) are normalised jump times, $Z_i$ are the jumps of the subordinator process, {arranged in non-increasing order},
  and $U_i \overset{iid}{\sim} \mathcal{N}(0, 1)$.
    \end{theorem}
    \begin{proof}
    Arguing as in~\cite[Section 7]{GeneralisedShotNoise}, we can extend~(\ref{eq:JumpFunctionNormalVarianceMean}) to any NVM process defined on $t\in[0, T]$, with $\tilde{V}_i = TV_i\sim\mathcal{U}[0,T]$, to obtain, for all $u\in[0,T]$,
    \begin{align*}
        X(u) &= \sum_{i=1}^{\infty}\left[\mu_{W}Z_{i}+\sigma_{W}\sqrt{Z_{i}}U_{i}\right]\mathbbm{1}(\tilde{V}_{i}\leq u)\\
        &= \mu_{W}\sum_{i=1}^{\infty}Z_{i}\mathbbm{1}(\tilde{V}_{i}\leq u) + \sigma_{W}\sum_{i=1}^{\infty}\sqrt{Z_{i}}U_{i}\mathbbm{1}(\tilde{V}_{i}\leq u)\\
        &= \mu_{W}M(u) + \sigma_{W}S(u),
    \end{align*}
    where \(M(u)\) is a subordinator L\'evy process and \(S(u)\) a symmetric Gaussian mixture process. 
    Therefore,
    \begin{equation*}
        dX(u) = \mu_{W}dM(u) + \sigma_{W}dS(u),
    \end{equation*}
    and hence we obtain the following representation for \(\mbold{I}({\textit{\textbf{f}}_{t}})\):
    \begin{equation*}
       \mbold{I}({\textit{\textbf{f}}_{t}}) = \mu_{W}\int_{0}^{T}\textit{\textbf{f}}_{t}(u)dM(u) + \sigma_{W}\int_{0}^{T}\textit{\textbf{f}}_{t}(u)dS(u).
    \end{equation*}
    From~\cite[Corollary 8.2]{LevyInFinance}, a stochastic integral with respect to a Lévy subordinator admits
    the a.s.\ generalised-shot noise representation,
    \begin{equation*}
        \int_{0}^{T}\textit{\textbf{f}}_{t}(u)dM(u) = \sum_{i=1}^{\infty}Z_{i}\textit{\textbf{f}}_{t}(\tilde{V}_{i})\mathbbm{1}(\tilde{V}_{i} \leq t).
    \end{equation*}
    Similarly, Rosinski proves in~\cite[Section~4]{GenShotNoiseStochIntegral} the a.s.\ representation of stochastic integrals with respect to Lévy processes of type G. {These are symmetric normal variance mixture processes such as the symmetric Student-t, symmetric $\alpha$-stable and Laplace which are special cases of NVM L\'{e}vy processes having $\mu_W=0$}, in the form:
    \begin{equation*}
        \int_{0}^{T}\textit{\textbf{f}}_{t}(u)dS(u) = \sum_{i=1}^{\infty}\sqrt{Z_{i}}U_{i}\textit{\textbf{f}}_{t}(\tilde{V}_{i})\mathbbm{1}(\tilde{V}_{i} \leq t).
    \end{equation*}
    Combining the last three expressions proves the claimed result. 
    \end{proof}
    Applying the result of the theorem to $X(t)$ in~(\ref{eq:SDE_solution}), with,
    $\textit{\textbf{f}}_{t}(u) = e^{\mbold{A}(t-u)}\mbold{h}\mathbbm{1}(s\leq u \leq t),$
    yields:
    $$
       \mbold{I}\left(\textit{\textbf{f}}_{t}\right) = \sum_{i=1}^{\infty}\left[\mu_{W}Z_{i}+\sigma_{W}\sqrt{Z_{i}}U_{i}\right]e^{\mbold{A}(t-\tilde{V}_{i})}\mbold{h}{\mathbbm{1}(\tilde{V}_i\in (s,t ])}
    $$
{ with $\Tilde{V}_{i} \overset{iid}{\sim} \mathcal{U}(0,T]$ as before.}
   
    \section{Convergence of residual SDE to a Gaussian SDE}\label{section:SDEConvergence}
    In this section we prove that the residual series of the truncated shot noise representation of the SDE  
    in the previous section, converges to 
    Brownian motion-driven SDE, as the truncation level $\epsilon\downarrow 0$. 
    Employing as before random Poisson truncations of the subordinator jumps $Z_i$, we can write,
    \begin{equation*}
        \boldsymbol{X}(t) = e^{\mbold{A}(t-s)}\boldsymbol{X}(s) + \boldsymbol{Z}_{\epsilon}(s,t) + \boldsymbol{R}_{\epsilon}(s,t),
    \end{equation*}
    where $\boldsymbol{I(f_{t})} = \boldsymbol{Z}_{\epsilon}(s,t) + \boldsymbol{R}_{\epsilon}(s,t)$, with,
    \begin{align*}
        &\boldsymbol{Z}_{\epsilon}(s,t) = \sum_{i: Z_{i} > \epsilon}\left[\mu_{W}Z_{i}+\sigma_{W}\sqrt{Z_{i}}U_{i}\right]e^{\mbold{A}(t-\tilde{V}_{i})}\mbold{h}{\mathbbm{1}(\tilde{V}_i\in (s,t ])},\\
        &\boldsymbol{R}_{\epsilon}(s,t) = \sum_{i: Z_{i} \leq \epsilon}\left[\mu_{W}Z_{i}+\sigma_{W}\sqrt{Z_{i}}U_{i}\right]e^{\mbold{A}(t-\tilde{V}_{i})}\mbold{h}{\mathbbm{1}(\tilde{V}_i\in (s,t ])},
    \end{align*}
    and \(\Tilde{V}_{i} \sim {\mathcal{U}(0,T]}  \). 
    Here,
    $\boldsymbol{R}_{\epsilon}(s,t)$ is the 
    residual series driven by small jumps $Z_i < \epsilon$, which will be approximated by a Brownian-driven SDE with matched moments. 
    Theorem \ref{theorem:NVMConvergenceGaussian} and the results in Section~\ref{section:ExampleCases} are the starting
    point of the proof of this approximation.
    
    We first present a Lemma concerning NVM L\'{e}vy processes.
    \begin{lemma}
    \label{lemma:FournierNVM}
    Let \((Y_{\epsilon}(t))\) denote the standardised residual NVM L\'{e}vy process as defined as in Theorem \ref{theorem:NVMConvergenceGaussian}, and let \((B(t))\) be a standard Brownian motion. Then there exists a coupling \((Y_{\epsilon}(t),B(t))\) such that, under the conditions of Theorem~\ref{theorem:NVMConvergenceGaussian}, we have,
    \begin{align*}
        \mathbbm{E} \left[\sup_{t\in[0,T]}\left|Y_{\epsilon}(t)-B(t)\right|^{2}\right] \leq C_{T}\max\{S^{1/2}_{\epsilon}, S^{1/3}_{\epsilon}\},
    \end{align*}
    where \(C_{T}\) is a constant independent of \(\epsilon\) and:
        $$
        S_{\epsilon} = \mu_{W}^{4}\frac{M_{Z_{\epsilon}}^{(4)}}{\sigma_{\epsilon}^{4}}+6\mu_{W}^{2}\sigma_{W}^{2}\frac{M_{Z_{\epsilon}}^{(3)}}{\sigma_{\epsilon}^{4}}+ 3\sigma_{W}^{4}\frac{M_{Z_{\epsilon}}^{(2)}}{\sigma_{\epsilon}^{4}}.
        $$
    \end{lemma}
    \begin{proof}
       The process \((Y_{\epsilon}(t))\) is square-integrable, with drift parameter and diffusion coefficient $a=b=0$. Write $Q_X$ for its Lévy measure, and define:
       $$
            m_{n}(Q_{X}):= \int_{\mathbbm{R}^{*}}|x|^{n}Q_{X}(dx), \quad n\geq 1.
       $$ 
             The corresponding subordinator has Lévy measure $Q_Z$, satisfying \(Q_{Z}(\{z > \epsilon\}) = 0\).
             Note that \(m_{2}(Q_{X}) = 1\). Following Theorem \ref{theorem:NVMConvergenceGaussian}, the subordinator has no drift or diffusion component, so neither does \((Y_{\epsilon}(t))\). 
            From ~\cite[Theorem 3.1]{Fournier2009}, if \(m_{4}(Q_{X}) < \infty\), we have directly, 
            \begin{align*}
                \mathbbm{E} \left[\sup_{[0,T]}\left|Y_{\epsilon}(t)-B(t)\right|^{2}\right] \leq C_{T}\max\{m_{4}(Q_{X})^{1/2}, m_{4}(Q_{X})^{1/3}\},
            \end{align*}
            where \(C_{T}\) depends on \(T\) only. 
            Finally, writing \(\sigma_{\epsilon}^{2}\) for the variance of the NVM residual process, \(m_{4}(Q_{X})\) can be expresses as:
            \begin{align}
                m_{4}(Q_{X}) &= \int_{-\infty}^{\infty}(x/\sigma_{\epsilon})^{4}\int_{0}^{\infty}\mathcal{N}(dx;\mu_{W}z, \sigma_{W}^{2}z)Q_{Z}(dz)\nonumber\\
                &= \frac{1}{\sigma_{\epsilon}^{4}}\int_{0}^{\infty}Q_{Z}(dz)\int_{-\infty}^{\infty}x^{4}\mathcal{N}(dx;\mu_{W}z, \sigma_{W}^{2}z)dx\nonumber\\
                &= \frac{1}{\sigma_{\epsilon}^{4}}\int_{0}^{\infty}\left[\mu_{W}^{4}z^{4}+6\mu_{W}^{2}\sigma_{W}^{2}z^{3} + 3\sigma_{W}^{4}z^{2}\right]Q_{Z}(dz)\nonumber\\
                &= \mu_{W}^{4}\frac{M_{Z_{\epsilon}}^{(4)}}{\sigma_{\epsilon}^{4}}+6\mu_{W}^{2}\sigma_{W}^{2}\frac{M_{Z_{\epsilon}}^{(3)}}{\sigma_{\epsilon}^{4}}+ 3\sigma_{W}^{4}\frac{M_{Z_{\epsilon}}^{(2)}}{\sigma_{\epsilon}^{4}}. \nonumber
            \end{align}
        Substituting this into the earlier bound yields the required result.
    \end{proof}
    
    \begin{theorem}
    \label{theorem:StochIntegralConvergence}
    Let \((X_{\epsilon}(t))\) denote the residual NVM L\'evy process as in Theorem~\ref{theorem:NVMConvergenceGaussian}, and write \((\mathcal{F}_{t})\) for the filtration it generates. Let \(\textit{\textbf{f}}_{t}(u)\) denote some \(\mathcal{F}_{t}\)-previsible, \(L_{2}\)-integrable function, and define the process \((\mbold{R}_{\epsilon}(t))\) via:
    \begin{align*}
        d\mbold{R}_{\epsilon}(u) = \textit{\textbf{f}}_{t}(u)dX_{\epsilon}(u).
    \end{align*}
    If $(X_{\epsilon}(t))$ satisfies the conditions of Theorem~\ref{theorem:NVMConvergenceGaussian}, then
    \((\mbold{Y}_{\epsilon}(t))\),
    the \(\epsilon\)-normalised version of \((\mbold{R}_{\epsilon}(t))\),  satisfies:
    \begin{align*}
        \lim_{\epsilon\rightarrow 0}\mathbb{E} \left[\sup_{t\in[0,T]}\left\|\mbold{Y}_{\epsilon}(t)-\mbold{B}(t)\right\|_{1}^{2}\right] = 0.
    \end{align*}
{ where,  centering and normalising \((X_{\epsilon}(t))\) to obtain \((Y_{\epsilon}(t))\) as before, we define
        \begin{align*}
             \mbold{Y}_{\epsilon}(t) = \int_{0}^{t}\textit{\textbf{f}}_{t}(u)dY_{\epsilon}(u),
        \end{align*}
        and also a moment matched Gaussian process $(\mbold{B}(t))$ driven by standard Brownian motion $B(t)$ that can be coupled with $(Y_{\epsilon}(t))$ in the sense of Lemma \ref{lemma:FournierNVM},
        \begin{align*}
            \mbold{B}(t) = \int_{0}^{t}\textit{\textbf{f}}_{t}(u)dB(u)\,.
        \end{align*}
       }

    In particular, for each fixed time $t$, \(\mbold{Y}_{\epsilon}(t)\) converges in distribution to the Gaussian law:
    \begin{equation*}
        \mathcal{N} \left(\int_{0}^{t}\textit{\textbf{f}}_{t}(u)du,\int_{0}^{t}\textit{\textbf{f}}_{t}(u)\textit{\textbf{f}}_{t}(u)^{T}du\right).
    \end{equation*}
    \end{theorem}
    \begin{proof}
        Consider the SDE,
        \begin{align*}
            d\mbold{R}_{\epsilon}(u) = \textit{\textbf{f}}_{t}(u)dX_{\epsilon}(u),
        \end{align*}
        with solution given by,
        \begin{align*}
            \mbold{R}_{\epsilon}(t) = \int_{0}^{t}\textit{\textbf{f}}_{t}(u)dX_{\epsilon}(u).
        \end{align*}
        Since \(\left\|\textit{\textbf{f}}_{t}\right\|_{L_{2}}<\infty\) and \((X_{\epsilon}(t))\) is a semi-martingale, 
        its quadratic variation is well-defined; see Appendix~\ref{appendix:AppendixSDEMoments}. Therefore,
        we can compute the mean and variance of \(\mbold{R}_{\epsilon}(t)\) as,
        $$
            \mathbbm{E}\left[\mbold{R}_{\epsilon}(t)\right] = \mu_{W}M_{Z_{\epsilon}}^{(1)}\int_{0}^{t}\textit{\textbf{f}}_{t}(u)du,\quad
            \mathbbm{Var}\left[\mbold{R}_{\epsilon}(t)\right] = \left[\mu_{W}^{2}M_{Z_{\epsilon}}^{(2)}+\sigma_{W}^{2}M^{(1)}_{Z_{\epsilon}}\right]\int_{0}^{t}\textit{\textbf{f}}_{t}(u)\textit{\textbf{f}}_{t}(u)^{T}du. 
        $$
        Centering and normalising \((X_{\epsilon}(t))\) to obtain \((Y_{\epsilon}(t))\) as before, we now consider the process,
        \begin{align*}
             \mbold{Y}_{\epsilon}(t) = \int_{0}^{t}\textit{\textbf{f}}_{t}(u)dY_{\epsilon}(u),
        \end{align*}
        along with a matching process driven by standard Brownian motion,
        \begin{align*}
            \mbold{B}(t) = \int_{0}^{t}\textit{\textbf{f}}_{t}(u)dB(u),
        \end{align*}
        where,
        \begin{equation*}
            \expec{\mbold{B}(t)} = \int_{0}^{t}\textit{\textbf{f}}_{t}(u)du, \hspace{1.6cm} \mathbbm{Var}\left[\mbold{B}(t)\right] = \int_{0}^{t}\textit{\textbf{f}}_{t}(u)\textit{\textbf{f}}_{t}(u)^{T}du.
        \end{equation*}
        Letting \(T(t) = Y_{\epsilon}(t)-B(t)\), we have,
        $$
            \left\|\mbold{Y}_{\epsilon}(t)-\mbold{B}(t)\right\|_{1} = \left\|\int_{0}^{t}\textit{\textbf{f}}_{t}(u)dT(u)\right\|_{1} 
            \leq \sup_{u\in[0,t]}||\textit{\textbf{f}}_{t}(u)||_{1}\left|T(t)\right|
            = \sup_{u\in[0,t]}||\textit{\textbf{f}}_{t}(u)||_{1}|Y_{\epsilon}(t)-B(t)|,
        $$
        so that,
        \begin{align*}
             \sup_{t\in[0,T]}\left\|\mbold{Y}_{\epsilon}(t)-\mbold{B}(t)\right\|_{1}^{2} \leq \sup_{t\in[0,T]}\sup_{u\in[0,t]}||\textit{\textbf{f}}_{t}(u)||_{1}^{2}\sup_{t\in[0,T]}|Y_{\epsilon}(t)-B(t)|^{2},
        \end{align*}
{and hence, applying Lemma \ref{lemma:FournierNVM}, there exist coupled processes $(Y_{\epsilon}(t))$  and $(B(t))$ such that:}
        \begin{align*}
            \mathbb{E} \left[\sup_{t\in[0,T]}\left\|\mbold{Y}_{\epsilon}(t)-\mbold{B}(t)\right\|_{1}^{2}\right]&\leq \sup_{t\in[0,T]}\sup_{u\in[0,t]}||\textit{\textbf{f}}_{t}(u)||_{1}^{2}\mathbbm{E}\left[\sup_{t\in[0,T]}|Y_{\epsilon}(t)-B(t)|^{2}\right]\\
            &\leq C_{T}\sup_{t\in[0,T]}\sup_{u\in[0,t]}\left\|\textit{\textbf{f}}_{t}(u)\right\|_{1}^{2}\max\{S_{\epsilon}^{1/2}, S_{\epsilon}^{1/3}\}.
        \end{align*}
        By our assumptions on $\textit{\textbf{f}}_{t}(u), u \in [0, t]$, and under condition~(\ref{eq:NecessaryAndSufficient}) of Theorem~\ref{theorem:NVMConvergenceGaussian},
        we have \(\lim_{\epsilon \rightarrow 0} S_{\epsilon} = 0\), and since \(C_{T}\) is independent of \(\epsilon\),
        \begin{align*}
            \lim_{\epsilon\rightarrow 0}\mathbb{E} \left[\sup_{t\in[0,T]}\left\|\mbold{Y}_{\epsilon}(t)-\mbold{B}(t)\right\|_{1}^{2}\right] = 0.
        \end{align*}
        completing the proof.
        \end{proof}

    Building on the results in Sections~\ref{section:NGProof},~\ref{section:NTSProof}, and~\ref{section:GHProof}, Theorem~\ref{theorem:StochIntegralConvergence} proves the Gaussian representation of \(\boldsymbol{R}_{\epsilon}(t)\) when
    the underlying process is known to converge to a diffusion.
    
    In Sections~\ref{section:NTSProof} and~\ref{section:GHProof} we justified the Gaussian convergence of the NTS and the GH process residuals, respectively, and therefore, Theorem~\ref{theorem:StochIntegralConvergence} justifies the Gaussian representation of \((\boldsymbol{R}_{\epsilon}(t))\). While an exact expression for \(M^{(1)}_{Z_{\epsilon}}\) exists in the NTS case, we can only provide small-argument bounds for the GIG case; see Appendix~\ref{appendix:NVMMoments}.
    
    In Section~\ref{section:NGProof} we showed that the NG residual process fails to converge to Brownian motion as \(\epsilon\to 0\). Therefore, there is no mathematically justified reason to model \(\boldsymbol{R}_{\epsilon}(t)\) as a diffusion process there. While the large concentration of residual jumps near \(0\) shown in Figure~\ref{fig:NGCLTVerification1} suggests it might be reasonable to set \(\boldsymbol{R}_{\epsilon}(t) = 0\), we could alternatively model the residual process via its mean, such that \(\boldsymbol{\hat{R}}_{\epsilon}(t) = \mathbbm{E}[\boldsymbol{\hat{R}}_{\epsilon}(t)]\), as suggested in \cite{SmallJumps}.
    
\section{Conclusions}
    In this work, new theoretical properties of NVM processes were established, motivated by 
    the desire to investigate the application of some of the Bayesian inference algorithms introduced in~\cite{LevyStateSpaceModel}
    to state space models driven by more general classes of non-Gaussian noise. We identified natural sufficient conditions 
    that guarantee the Gaussian convergence of the error process associated with an NVM process subjected to random Poisson truncations of their shot-noise series representations. We also showed that this Gaussian convergence does not always occur, and provided sufficient conditions
    for cases when it fails. 
    
    Moreover, under the same Gaussian-convergence conditions, we established the process-level convergence of a family of associated stochastic integrals, thus justifying the Gaussian representation of the residuals of these integrals.
    In Section~\ref{section:ExampleCases} we showed that Brownian motion with drift subordinated to the residual processes of a TS or a GIG-type
    L\'{e}vy process converges to a Wiener process. Furthermore, in Section~\ref{section:LevySSModel} we established the validity
    of the Gaussian representation of the residual of the stochastic integral with respect to the NTS and GH processes (excluding the Normal-gamma edge case). 
    
    Subordination to a Gamma process is shown not to converge to a Gaussian limit. Therefore,  the residuals of stochastic integrals with respect to NG processes cannot be represented by Gaussians. One alternative direction would be to investigate whether fitting a Gaussian to the residual would still yield improved accuracy in Bayesian inference procedures such as particle filtering. A more interesting possibility would be to explore the distribution to which the NG residual converges, as, for any \(\epsilon = h(\tilde{\epsilon}) > 0\), in Section~\ref{section:NGProof} we showed that the residual is non-zero and heavier-tailed than the Gaussian.
    
    Finally, we note that the current analysis applies to one-dimensional Lévy processes and the corresponding linear SDEs. An extension of the methodology to multivariate Lévy processes would allow generalising the `Lévy State Space Model'~\cite{LevyStateSpaceModel} and the 
    associated methodological tools to more general state-space models, importantly allowing for more sophisticated  modelling of spatial dependencies in high dimensional data.

\vspace*{0.3in}

\appendix

\centerline{\Large\bf Appendix}
\section{Upper and lower bounds on Jaeger integrals\label{appendix:UpperLowerJaegerBounds}}
 {Define first the Jaeger integral \cite{Freitas2018} as parameterised in \cite{GIGLevy} (Section 3):
\[J(z)=\int_0^\infty\frac{e^{-\frac{x^2z}{2\delta^2}}}{x|H_{|\lambda|}(x)|^2}dx.\]}
The integrand in the GIG Lévy measure in~(\ref{eq:GIGMeasure}) depends on the value of \(|\lambda|\)~\cite{GIGLevy}. We first consider the region \(|\lambda| \leq \frac{1}{2}\). From~\cite[Theorem~3]{GIGLevy}, a suitable upper bound on \(\left[{z|H_{|\lambda|}(z)|^{2}}\right]^{-1}\) for \(|\lambda| \leq \frac{1}{2}\) is given by,
    \begin{equation*}
        \frac{1}{z|H_{|\lambda|}(z)|^{2}} \leq \frac{1}{B(z)},
    \end{equation*}
    where,
    \begin{equation*}
        \frac{1}{B(z)} = 
        \begin{cases}
        \frac{1}{H_{0}}\left(\frac{z}{z_{0}}\right)^{2|\lambda|-1}, \ & z < z_{0},\\
        \frac{1}{H_{0}}, \ & z \geq z_{0},
        \end{cases}
    \end{equation*}
    with $z_{0} \in (0, \infty)$ and $H_{0} = z_{0}|H_{|\lambda|}(z_{0})|^{2}$. This leads to the upper bound:
    \begin{align*}
        J(z) &\leq \frac{1}{H_{0}z_{0}^{2|\lambda|-1}}\int_{0}^{z_{0}}y^{2|\lambda|-1} \exp \left(-\frac{zy^{2}}{2\delta^{2}}\right)dy + \frac{1}{H_{0}}\int_{z_{0}}^{\infty}\exp \left(-\frac{zy^{2}}{2\delta^{2}}\right)dy\\
        &=\frac{1}{H_{0}z_{0}^{2|\lambda|-1}}\left[\frac{\delta^{2|\lambda|}2^{|\lambda|-1}}{z^{|\lambda|}}\gamma\left(|\lambda|, \frac{z_{0}^{2}}{2\delta^{2}}z\right)\right]+\frac{1}{H_{0}}\left[\frac{\delta}{\sqrt{2}}z^{-\frac{1}{2}}\Gamma\left(\frac{1}{2},\frac{z_{0}^{2}}{2\delta^{2}}z\right)\right].
    \end{align*}
    Recall the series expansion for $\gamma(s,x)$ from Section~\ref{section:NGProof} and also that
    for $\Gamma(s,x)$ we have~\cite{UpperIncomplete}, when \(|\lambda| \neq 0\):
        $$\Gamma(s,x) = \Gamma(s) - \sum_{n=0}^{\infty}(-1)^{n}\frac{x^{s+n}}{n!(s+n)}.
        $$
    Combining these with the previous bound:
    \begin{align}
        J(z) 
        &\leq 
        \frac{1}{H_{0}z_{0}^{2|\lambda|-1}}\left\{\frac{z_{0}^{2|\lambda|}}{2}\left[\frac{1}{|\lambda|}-\frac{z_{0}^{2}}{2\delta^{2}}\frac{1}{|\lambda|+1}z+\mathcal{O}(z^{2})\right]\right\}\nonumber\\
        &\quad
        +\frac{\delta}{H_0\sqrt{2}}\left\{\sqrt{\pi} z^{-\frac{1}{2}} - \frac{2z_{0}z^{\frac{1}{2}}}{\delta\sqrt{2}}z^{-\frac{1}{2}}+\frac{z_{0}^{3}z^{\frac{3}{2}}}{3\delta^{3}\sqrt{2}}z^{-\frac{1}{2}} +\mathcal{O}(z^{2}) \right\}\nonumber\\
        &= 
        \frac{1}{H_{0}z_{0}^{2|\lambda|-1}}\left\{\frac{z_{0}^{2|\lambda|}}{2}\left[\frac{1}{|\lambda|}+\mathcal{O}(z)\right]\right\}+\frac{\delta}{H_{0}}\sqrt{\frac{\pi}{2}}z^{-\frac{1}{2}} - \frac{z_{0}}{H_{0}}+\mathcal{O}(z)\nonumber\\
        &= 
        \frac{\delta}{H_{0}}\sqrt{\frac{\pi}{2}}z^{-\frac{1}{2}} + \frac{z_{0}}{H_{0}}\left(\frac{1}{2|\lambda|}-1\right) + \mathcal{O}(z).\nonumber
    \end{align}
    Finally, we refer to~\cite[Theorem~1]{GIGLevy} for a lower bound on the Jaeger integral for \(|\lambda|\leq \frac{1}{2}\):
    \begin{equation*}
        J(z) = \int_{0}^{\infty}\frac{1}{y|H_{|\lambda|}(y)|^{2}}\exp \Big(-\frac{zy^{2}}{2\delta^{2}}\Big)dy \geq \frac{\pi}{2}\int_{0}^{\infty}\exp \Big(-\frac{zy^{2}}{2\delta^{2}}\Big)dy
        = \frac{\delta\pi}{2}\sqrt{\frac{\pi}{2}}z^{-\frac{1}{2}}.
    \end{equation*}
    
    When \(|\lambda| > \frac{1}{2}\), \(y|H_{|\lambda|}(y)|^{2}\) is non-increasing rather than non-decreasing,
    and the relevant bounds become:
    \begin{align}
        \int_{0}^{\infty}\frac{1}{y|H_{|\lambda|}(y)|^{2}}\exp \Big(-\frac{zy^{2}}{2\delta^{2}}\Big)dy \leq \frac{\pi}{2}\int_{0}^{\infty}\exp \Big(-\frac{zy^{2}}{2\delta^{2}}\Big)dy &= \frac{\delta\pi}{2}\sqrt{\frac{\pi}{2}}z^{-\frac{1}{2}}. \label{Eq:bound_J}
    \end{align}
    From \cite[Theorem~3]{GIGLevy}, a suitable lower bound on $z|H_{|\lambda|}(z)|^{2}]^{-1}$ for \(|\lambda| \geq  \frac{1}{2}\) is given by, 
    \begin{equation*}
        \frac{1}{z|H_{|\lambda|}(z)|^{2}} \geq \frac{1}{B(z)},
    \end{equation*}
    where \(B(z)\) is defined as above, from which we deduce that for small \(z \in [0, \epsilon]\) and \(|\lambda| \leq \frac{1}{2}\):
    \begin{equation*}
        \delta \left(\frac{\pi}{2}\right)^{\frac{3}{2}}z^{-\frac{1}{2}} \leq J(z) \leq \frac{\delta}{H_{0}}\left(\frac{\pi}{2}\right)^{\frac{1}{2}}z^{-\frac{1}{2}} + \frac{z_{0}}{H_{0}}\left(\frac{1}{2|\lambda|}-1\right) +\mathcal{O}(z).
    \end{equation*}
    And similarly, if \(|\lambda|\geq \frac{1}{2}\):
    \begin{equation*}
        \frac{\delta}{H_{0}}\left(\frac{\pi}{2}\right)^{\frac{1}{2}}z^{-\frac{1}{2}} + \frac{z_{0}}{H_{0}}\left(\frac{1}{2|\lambda|}-1\right) +\mathcal{O}(z) \leq J(z)  \leq \delta \left(\frac{\pi}{2}\right)^{\frac{3}{2}}z^{-\frac{1}{2}}.
    \end{equation*}

    \section{Derivation of SDE moments\label{appendix:AppendixSDEMoments}}
    Recall the definition of \(\mbold{R}_{\epsilon}(t)\) as: 
    \begin{align*}
        \mbold{R}_{\epsilon}(t) = \int_{0}^{t}\textit{\textbf{f}}_{t}(u)dX^{\epsilon}_{u}.
    \end{align*}
    Since \((X_{t}^{\epsilon})\) is a semi-martingale, it is clear that \(\Tilde{X}_{t}^{\epsilon} = X_{t}^{\epsilon} - t\mu_{W}M_{Z_{\epsilon}}^{(1)}\) is a martingale, and:
    \begin{align*}
        \mbold{R}_{\epsilon}(t) = \mu_{W}M_{Z_{\epsilon}}^{(1)}\int_{0}^{t}\textit{\textbf{f}}_{t}(u)du + \int_{0}^{t}\textit{\textbf{f}}_{t}(u)d\Tilde{X}^{\epsilon}_{u},
    \end{align*}
    Then, clearly, we have,
    $$
        \expec{\mbold{R}_{\epsilon}(t)} = \mu_{W}M_{Z_{\epsilon}}^{(1)}\int_{0}^{t}\textit{\textbf{f}}_{t}(u)du,
    $$
    and the It\^{o} isometry yields,
    \begin{align*}
        \mathbbm{Var}\left[\mbold{R}_{\epsilon}(t)\right] = \expec{\left(\int_{0}^{t}\textit{\textbf{f}}_{t}(u)d\Tilde{X}^{\epsilon}_{u}\right)\left(\int_{0}^{t}\textit{\textbf{f}}_{t}(u)d\Tilde{X}^{\epsilon}_{u}\right)^{T}} = 
        \expec{\int_{0}^{t}\textit{\textbf{f}}_{t}(u)\textit{\textbf{f}}_{t}(u)^{T}d\big[\Tilde{X}^{\epsilon}\big]_{u}},
    \end{align*}
    where \(\big[\Tilde{X}^{\epsilon}\big]_{t}\) is the quadratic variation of the compensated Lévy process, with expectation~\cite{FinancialModellingCont}:
    \begin{align*}
        \expec{d\left[\Tilde{X}^{\epsilon}\right]_{t}}= \expec{\int_{\mathbbm{R} \backslash \{0\}}x^{2}{N}(dt,dx)} = \left[\mu_{W}^{2}M_{Z_{\epsilon}}^{(2)}+\sigma_{W}^{2}M_{Z_{\epsilon}}^{(1)}\right]dt.
    \end{align*}
    Finally, the corresponding expression for the variance of \(\mbold{R}_{\epsilon}(t)\) is also easily obtained as:
    $$
        \mathbbm{Var}\left[\mbold{R}_{\epsilon}(t)\right] = \left[\mu_{W}^{2}M_{Z_{\epsilon}}^{(2)}+\sigma_{W}^{2}M_{Z_{\epsilon}}^{(1)}\right]\int_{0}^{t}\textit{\textbf{f}}_{t}(u)\textit{\textbf{f}}_{t}(u)^{T}du.
    $$
    \section{Centering for NVM processes}
\label{app:subordinator} 
The following lemma is probably known but, as we could not easily locate a specific reference, we provide
a proof for completeness. Recall the convergent generalised shot-noise representation of a L\'{e}vy process
    from Section~\ref{section:BackgroundGeneralisedShotNoise}:
        \begin{equation*}
            X(t) = \sum_{i=1}^{\infty}H(Z_i, U_{i})\mathbbm{1}(V_{i} \leq t) -tb_{i}, \ 0 \leq t \leq T
        \end{equation*}
\begin{lemma}
For an NVM L\'{e}vy process $(X(t))$, the following generalised shot noise representation converges a.s.:
        \begin{equation*}
            X(t) = \sum_{i=1}^{\infty}H(Z_i, U_{i})\mathbbm{1}(V_{i} \leq t), \quad {0\leq t\leq T}.
        \end{equation*}
\end{lemma}

\begin{proof}
Recall the form of the Lévy measure of NVM processes from~(\ref{eq:GeneralNormalVarianceMeanMeasure2}).
    An NVM process need not be compensated, and hence we may take $b_i=0$ for all $i$ in the shot noise representation, provided,
    \begin{equation}
        \int_{-\infty}^{\infty}\left(1\wedge|x|\right)Q_{X}(dx) = \int_{-\infty}^{\infty}\left(1\wedge|x|\right)\int_{0}^{\infty}\mathcal{N}\left(x;\mu_{W}z,\sigma_{W}^{2}z\right)Q_{Z}(dz)dx < \infty.\label{eq:VMG_condition}
    \end{equation}
    Since by the L\'{e}vy measure definition $\int_{-\infty}^{\infty}\left(1\wedge x^{2}\right)Q_{X}(dx)$ is finite, we must also have that 
    $\int_{\left\{|x|>1\right\}}Q_{X}(dx)$ is finite. So, concentrating on the interval $|x|\leq1$:
    \begin{align}
        I &:= \int_{|x|\leq1}|x|\int_{0}^{\infty}\mathcal{N}\left(x;\mu_{W}z,\sigma_{W}^{2}z\right)Q_{Z}(dz)dx\nonumber\\
        &=\int_{|x|\leq1}\frac{|x|}{\sqrt{2\pi\sigma_{W}^{2}}}\int_{0}^{\infty}z^{-\frac{1}{2}}\exp\left[-\frac{1}{2\sigma_{W}^{2}z}\left(x-\mu_{W}z\right)^{2}\right]Q_{Z}(dz)dx\nonumber\\
        &= \int_{|x|\leq1}\frac{|x|}{\sqrt{2\pi\sigma_{W}^{2}}}\exp\left(\frac{x\mu_{W}}{\sigma_{W}^{2}}\right)\int_{0}^{\infty}z^{-\frac{1}{2}}\exp\left(-\frac{x^{2}}{2\sigma_{W}^{2}}z^{-1}-\frac{\mu_{W}^{2}}{2\sigma_{W}^{2}}z\right)Q_{Z}(dz)dx\nonumber \\
        &\leq \int_{|x|\leq1}\frac{|x|}{\sqrt{2\pi\sigma_{W}^{2}}}\exp\left(\frac{\mu_{W}}{\sigma_{W}^{2}}\right)\int_{0}^{\infty}z^{-\frac{1}{2}}\exp\left(-\frac{x^{2}}{2\sigma_{W}^{2}}z^{-1}\right)Q_{Z}(dz)dx.\label{eq:IntegralSmallJumpsCompensation}
    \end{align}
Note that the unimodal function $z^{-\frac{3}{2}}\exp\{-x^2/(2z\sigma_{W}^{2})\}$ achieves its maximum at $z=\frac{x^{2}}{3\sigma_{W}^{2}}$. Hence, the inner integrand may be bounded, for all $z\in (0,1]$ by,
\begin{align*}
    z^{-\frac{1}{2}}\exp\left(-\frac{x^{2}}{2\sigma_{W}^{2}}z^{-1}\right)\leq 
    \begin{cases} 
    z \frac{|x|^{-3}}{(\sqrt{3}\sigma_{W})^3}e^{-3/2},& |x|>3\sigma_W^2,\\
    z\exp(-\frac{x^{2}}{3\sigma_{W}^{2}}),&|x|\leq{3\sigma_{W}^{2}},
    \end{cases} 
\end{align*}
with the $|x|\leq{3\sigma_{W}^{2}}$ case corresponding to the supremum lying to the right of $z=1$. 
Therefore, the inner integral in~(\ref{eq:IntegralSmallJumpsCompensation}) over $z\in(0,1)$ can be bounded, 
using~(\ref{eq:SubordinatorRequirement}), by,
\begin{align*}
\int_{0}^{1}z^{-\frac{1}{2}}\exp\left(-\frac{x^{2}}{2\sigma_{W}^{2}}z^{-1}\right)Q_{Z}(dz)dx\leq \begin{cases}C\frac{|x|^{-3}}{(\sqrt{3}\sigma_{W})^3}e^{-3/2},& |x|>3\sigma_W^2\\
C\exp(-\frac{x^{2}}{3\sigma_{W}^{2}}),&|x|\leq{3\sigma_{W}^{2}},\end{cases}
\end{align*}
where $C=\int_0^1zQ_Z(dz)<\infty$ is a constant that does not depend on $x$ or $\epsilon$. 
Moreover, using~(\ref{eq:SubordinatorRequirement}) again, we obtain,
\begin{align*}
    \int_{1}^{\infty}z^{-\frac{1}{2}}\exp\left(-\frac{x^{2}}{2\sigma_{W}^{2}}z^{-1}\right)Q_{Z}(dz)\leq \int_{1}^{\infty}Q_{Z}(dz)=C'<\infty.
\end{align*}
Combining the above bounds yields,
\begin{align*}        
I \;\leq\;
        \frac{1}{\sqrt{2\pi\sigma_{W}^{2}}}\exp\left(\frac{\mu_{W}}{\sigma_{W}^{2}}\right)
        &
            \left[\int_{\left\{|x|\leq (3\sigma_W^2\wedge 1)\right\}}{|x|}C\exp\Big(-\frac{x^{2}}{3\sigma_{W}^{2}}\Big)dx\right.\\
        &
            \left.\;\;+\int_{\left\{(3\sigma_W^2\wedge 1)<  |x|\leq 1\right\}}{|x|}C\frac{|x|^{-3}}{\{\sqrt{3}\sigma_{W}\}^3}e^{-3/2}dx\right.\\
        & 
            \left.\;\;+\int_{\left\{|x|\geq 1\right\}}{|x|}C'dx\right]<\infty,
 \end{align*}
establishing~(\ref{eq:VMG_condition}) and confirming that compensation is not required for NVM L\'{e}vy processes.

Given that $I$ is finite, the result of the lemma will follow from~\cite[Theorem~4.1]{GeneralisedShotNoise}, once 
we establish that, with, 
\[
A(s):=\int_0^s \int_{|x|\leq 1}x\sigma(r;dx)dr=\int_0^s \int_{|x|\leq 1}x{\cal{N}}(x;h(r)\mu_W,h(r)\sigma_W^2)dxdr, \quad s>0,
\]
the limit $a:=\lim_{s\to\infty} A(s)$ exists and is finite.
In the definition of $A(s)$, the term
 $\sigma(\cdot;\cdot)$ denotes the kernel,
 $\sigma(h(r);F)={\mathbbm P}(H(h(r),U)\in F)$ for all measurable $F$,
 which in the present setting is a collection of Gaussian measures.
 Since $r=Q_Z([z,+\infty ))$, we have $dr=-Q_Z(dz)$, and hence:
\begin{align*}
    A(s)=\int_0^s \int_{|x|\leq 1}x{\cal{N}}\left(x;h(r)\mu_W,h(r)\sigma_W^2\right)dxdr=\int_{h(s)}^\infty \int_{|x|\leq 1}x{\cal{N}}\left(x;z\mu_W,z\sigma_W^2\right)dxQ_Z(dz).
\end{align*}
Since we already showed that, 
\begin{align*}
    \int_{|x|\leq1}|x|\int_{0}^{\infty}\mathcal{N}\left(x;\mu_{W}z,\sigma_{W}^{2}z\right)Q_{Z}(dz)dx=I<\infty,
\end{align*}
we can exchange the order of integration, and since $h(s)\to 0$ as $s\to\infty$ by definition,
it follows that the limit $a:=\lim_{s\to\infty} A(s)$ exists.
Finally, we have,
\begin{align*}
    \left|\lim_{s\rightarrow \infty}A(s)\right|&=\left|\int_{|x|\leq 1}x\int_{0}^{\infty}\mathcal{N}\left(x;\mu_{W}z,\sigma_{W}^{2}z\right)Q_{Z}(dz)dx\right|\\&\leq \int_{|x|\leq1}|x|\int_{0}^{\infty}\mathcal{N}\left(x;\mu_{W}z,\sigma_{W}^{2}z\right)Q_{Z}(dz)dx <\infty,
\end{align*}
showing that $a$ is finite and concluding the proof.
\end{proof}
    \section{Moments for example processes}\label{appendix:NVMMoments}
    Recalling the discussion at the end of Section~\ref{section:SDEConvergence} in connection with the `Lévy State Space 
    model'~\cite{LevyStateSpaceModel}, we provide here expressions for $M_{Z_{\epsilon}}^{(1)}$ and $M_{Z_{\epsilon}}^{(2)}$ for our example processes.

    For the NG process in Section \ref{section:NGProof}, we have,
    \begin{align*}
        M_{Z_{\epsilon}}^{(1)} = \int_{0}^{\epsilon}z\nu z^{-1}\exp\left(-\frac{1}{2}\gamma^{2}z\right)dz
        = \frac{2\nu}{\gamma^{2}}\gamma\left(1, \frac{1}{2}\gamma^{2}\epsilon\right),
    \end{align*}
    and,
    \begin{align*}
        M_{Z_{\epsilon}}^{(2)} = \int_{0}^{\epsilon}z^{2}\nu z^{-1}\exp\left(-\frac{1}{2}\gamma^{2}z\right)dz
        =\frac{4\nu}{\gamma^{4}}\gamma\left(2, \frac{1}{2}\gamma^{2}\epsilon\right).
    \end{align*}
    Similarly, for the NTS process in Section \ref{section:NTSProof},
    \begin{align*}
        M_{Z_{\epsilon}}^{(1)} = \int_{0}^{\epsilon}zAz^{-1-\kappa}\exp\left(-\frac{1}{2}\gamma^{\frac{1}{\kappa}}z\right)dz
        = A \gamma^{\frac{\kappa-1}{\kappa}}2^{1-\kappa}\gamma\left(1-\kappa,\frac{1}{2}\epsilon\gamma^{\frac{1}{\kappa}}\right),
    \end{align*}
    and,
    \begin{align*}
        M_{Z_{\epsilon}}^{(2)} = \int_{0}^{\epsilon}z^{2}Az^{-1-\kappa}\exp\left(-\frac{1}{2}\gamma^{\frac{1}{\kappa}}z\right)dz
        =A\gamma^{\frac{\kappa-2}{\kappa}}2^{2-\kappa}\gamma\left(2-\kappa,\frac{1}{2}\epsilon\gamma^{\frac{1}{\kappa}}\right).
    \end{align*}
    The intractability of the Jaeger integral prohibits the derivation of an analytical expression for the moments of the GH process. However, for sufficiently small truncation levels \(\epsilon\), the use of asymptotic moment expansions provides useful approximations. For now, we restrict our analysis to the parameter range \(|\lambda| \leq \frac{1}{2}\);  the range \(|\lambda| \geq \frac{1}{2}\) yields similar results.
    
    To obtain a lower bound on the expected value of the subordinator jumps, \(M_{Z_{\epsilon}}^{(1)}\), we use the bound \(\frac{1}{z|H_{|\lambda|}(z)|^{2}}\geq \frac{\pi}{2}\):
    \begin{align*}
         M_{Z_{\epsilon}}^{(1)} &= \int_{0}^{\epsilon}z\frac{e^{-\frac{\gamma^{2}z}{2}}}{z}\left[\max(0,\lambda) + \frac{2}{\pi^{2}}\int_{0}^{\infty}\frac{1}{z\left|H_{|\lambda|}(z)\right|^{2}}e^{-\frac{zy^{2}}{2\delta^{2}}}dy \right]dz\\
         &\geq \max(0,\lambda)\frac{2}{\gamma^{2}}\gamma\left(1, \frac{1}{2}\gamma^{2}\epsilon\right) + \frac{1}{\pi}\int_{0}^{\epsilon}e^{-\frac{\gamma^{2}}{2}z}\int_{0}^{\infty}e^{-\frac{zy^{2}}{2\delta^{2}}}dydz\\
         &= \frac{2\max(0,\lambda)}{\gamma^{2}}\gamma\left(1, \frac{1}{2}\gamma^{2}\epsilon\right) + \frac{\delta}{\gamma}{\rm erf}\left(\frac{\gamma\sqrt{\epsilon}}{\sqrt{2}}\right).
   \end{align*}
   Using the expansion for the ${\rm erf}$ function  in Section~\ref{section:NGProof} and the series expansion of the exponential function, yields the following lower bound,
    \begin{align}
        M_{Z_{\epsilon}}^{(1)} &\geq \frac{2\max(0,\lambda)}{\gamma^{2}}\gamma\left(1, \frac{1}{2}\gamma^{2}\epsilon\right) + \frac{\delta}{\gamma}{\rm erf}\left(\frac{\gamma\sqrt{\epsilon}}{\sqrt{2}}\right) \nonumber\\ &=\frac{2\max(0,\lambda)}{\gamma^{2}}\gamma\left(1, \frac{1}{2}\gamma^{2}\epsilon\right) + \frac{\delta}{\gamma}\frac{2}{\sqrt{\pi}}\sum_{n=0}^{\infty}{\frac{(-1)^{n}}{n!(2n+1)}\left(\frac{\gamma\sqrt{\epsilon}}{\sqrt{2}}\right)^{2n+1}},\nonumber
    \end{align}
    which, for $\epsilon\to 0$, is equal to:
    \begin{equation}
           \frac{\delta\sqrt{2}}{\sqrt{\pi}}\sqrt{\epsilon}+ \max(0,\lambda)\epsilon + \mathcal{O}\left(\epsilon^{\frac{3}{2}}\right).\label{eq:ExpectedGHEScLowerBound}
    \end{equation}
    A corresponding upper bound can be derived using the bound in Appendix~\ref{appendix:UpperLowerJaegerBounds}.
    We have,
    \begin{align}
        M_{Z_{\epsilon}}^{(1)}
         &=\max(0,\lambda)\int_{0}^{\epsilon}e^{-\frac{\gamma^{2}z}{2}}dz + \frac{2}{\pi^{2}}\int_{0}^{\epsilon}e^{-\frac{\gamma^{2}z}{2}}\int_{0}^{\infty}\frac{1}{z\left|H_{|\lambda|}(z)\right|^{2}}e^{-\frac{zy^{2}}{2\delta^{2}}}dydz\nonumber\\
         &\leq \frac{2}{\gamma^{2}}\max(0,\lambda)\gamma\left(1, \frac{1}{2}\gamma^{2}\epsilon\right) + \frac{2}{\pi^{2}}\int_{0}^{\epsilon}e^{-\frac{z\gamma^{2}}{2}}\left[\frac{\sqrt{\pi}\delta}{H_{0}\sqrt{2}}z^{-\frac{1}{2}} + \frac{z_{0}}{H_{0}}\left(\frac{1}{2|\lambda|}-1\right)\right]dz\nonumber\\
         &= \frac{2}{\gamma^{2}}\max(0,\lambda)\gamma\left(1, \frac{1}{2}\gamma^{2}\epsilon\right)+ \frac{\sqrt{2}\delta}{H_{0}\pi\sqrt{\pi}}\left[\frac{\sqrt{2\pi}}{\gamma}{\rm erf}\left(\frac{\sqrt{\epsilon}\gamma}{\sqrt{2}}\right)\right]+ \frac{4z_{0}}{\pi^{2}\gamma^{2}H_{0}}\left(\frac{1}{2|\lambda|}-1\right)\left(1-e^{-\frac{c\gamma^{2}}{2}}\right),
            \nonumber
    \end{align}
    where, as $\epsilon\to 0$, the last expression is equal to:
    \begin{equation}
         \frac{2\delta\sqrt{2}}{H_{0}\pi\sqrt{\pi}}\sqrt{\epsilon} + \left[\max(0,\lambda)+\frac{2z_{0}}{\pi^{2}H_{0}}\left(\frac{1}{2|\lambda|}-1\right)\right]\epsilon + \mathcal{O}\left(\epsilon^{\frac{3}{2}}\right). \label{eq:ExpectedGHEScUpperBound}
    \end{equation}
    Equations~(\ref{eq:ExpectedGHEScLowerBound}),~(\ref{eq:ExpectedGHEScUpperBound}) imply that, for small $\epsilon$, 
    $M_{Z_\epsilon}^{(1)}$ is approximately bounded above and below~by,
    \begin{equation*}
        \frac{\delta\sqrt{2}}{\sqrt{\pi}}\sqrt{\epsilon}
            \quad\mbox{and}\quad \frac{2\delta\sqrt{2}}{H_{0}\pi\sqrt{\pi}}\sqrt{\epsilon},\quad\mbox{respectively},
    \end{equation*}
    and we can therefore conclude that \(M_{Z_{\epsilon}}^{(1)} \sim \sqrt{\epsilon}\) as $\epsilon\to 0$,
    for \(\delta \neq 0\) and \(|\lambda| \leq \frac{1}{2}\).
    
    To characterise the behaviour of \(M_{Z_{\epsilon}}^{(2)}\), we again use the bound \(\frac{1}{z|H_{|\lambda|}(z)|^{2}}\geq \frac{\pi}{2}\), which gives,
    \begin{align}
        M_{Z_{\epsilon}}^{(2)} &= \max(0,\lambda)\int_{0}^{\epsilon}ze^{-\frac{\gamma^{2}}{2}z}dz+\frac{2}{\pi^{2}}\int_{0}^{\epsilon}ze^{-\frac{\gamma^{2}}{2}z}\int_{0}^{\infty}\frac{1}{z\left|H_{|\lambda|}(t)\right|^{2}}e^{-\frac{zy^{2}}{2\delta^{2}}}dydz \nonumber\\
        &\geq  \frac{4}{\gamma^{4}}\max(0,\lambda)\gamma\left(2, \frac{1}{2}\gamma^{2}\epsilon\right) + \frac{2\delta}{\gamma^{3}\sqrt{\pi}}\gamma\left(\frac{3}{2}, \frac{1}{2}\gamma^{2}\epsilon\right),\nonumber
    \end{align}
    where, as $\epsilon\to 0$, the last expression above is,
    \begin{equation}
    \frac{\delta}{\sqrt{2\pi}}\left(\frac{2}{3}\epsilon\sqrt{\epsilon}\right)+\frac{1}{2}\max(0,\lambda)\epsilon^{2} + \mathcal{O}\left(\epsilon^{\frac{5}{2}}\right)
        = \frac{\sqrt{2}\delta}{3\sqrt{\pi}}\epsilon\sqrt{\epsilon} + \frac{1}{2}\max(0,\lambda)\epsilon^{2} + \mathcal{O}(\epsilon^{\frac{5}{2}}).\label{eq:GIGVarSCLowerBound}
    \end{equation}
    For the corresponding upper bound we similarly have,
    \begin{align}
        M_{Z_{\epsilon}}^{(2)} &\leq\max(0,\lambda)\int_{0}^{\epsilon}ze^{-\frac{\gamma^{2}}{2}z}dz+ \frac{2}{\pi^{2}}\int_{0}^{\epsilon}ze^{-z\frac{\gamma^{2}}{2}}\left[\frac{\sqrt{\pi}\delta}{H_{0}\sqrt{2}}z^{-\frac{1}{2}} + \frac{z_{0}}{H_{0}}\left(\frac{1}{2|\lambda|}-1\right) \right]dt\nonumber\\
        &= \frac{4}{\gamma^{4}}\left[\max(0,\lambda)+\frac{2z_{0}}{\pi^{2}H_{0}}\left(\frac{1}{2|\lambda|}-1\right)\right]\gamma\left(2, \frac{1}{2}\gamma^{2}\epsilon\right) + \frac{\sqrt{2}\delta}{\pi\sqrt{\pi}H_{0}}\int_{0}^{\epsilon}t^{\frac{1}{2}}e^{-t\frac{\gamma^{2}}{2}}dt\nonumber\\
        &=\frac{4}{\gamma^{4}}\left[\max(0,\lambda)+\frac{2z_{0}}{\pi^{2}H_{0}}\left(\frac{1}{2|\lambda|}-1\right)\right]\gamma\left(2, \frac{1}{2}\gamma^{2}\epsilon\right) + \frac{4\delta}{\gamma^{3}\pi\sqrt{\pi}H_{0}}\gamma\left(\frac{3}{2}, \frac{1}{2}\gamma^{2}\epsilon\right),\nonumber
    \end{align}
    where the last expression above, for $\epsilon\to 0$, is equal to,
    \begin{equation}
        \frac{2\sqrt{2}\delta}{3\pi\sqrt{\pi}H_{0}}\epsilon\sqrt{\epsilon} +\frac{1}{2}\left[\max(0,\lambda)+\frac{2z_{0}}{\pi^{2}H_{0}}\left(\frac{1}{2|\lambda|}-1\right)\right]\epsilon^{2} + \mathcal{O}(\epsilon^{\frac{5}{2}}).\label{eq:GIGVarSCUpperBound} 
    \end{equation}
    From~(\ref{eq:GIGVarSCUpperBound}),~(\ref{eq:GIGVarSCLowerBound}) we have that, for small $\epsilon$,
        $M_{Z_\epsilon}^{(2)}$ is approximately bounded above and below by,
    \begin{equation*}
        \frac{\sqrt{2}\delta}{3\sqrt{\pi}}\epsilon\sqrt{\epsilon}
        \quad
            \mbox{and}
        \quad \frac{2\sqrt{2}\delta}{3\pi\sqrt{\pi}H_{0}}\epsilon\sqrt{\epsilon}, \quad\mbox{respectively},
    \end{equation*}
    and, therefore,
    \(M_{Z_{\epsilon}}^{(2)} \sim \epsilon\sqrt{\epsilon}\), for all \( \delta \neq 0, \left|\lambda\right|\leq \frac{1}{2}\ \).
    
    Finally, in the case \(\left|\lambda\right|\geq \frac{1}{2}\), the upper and lower bounds on \(\frac{1}{z|H_{|\lambda|}(z)|^{2}}\) are reversed, and hence so are the bounds in \((\ref{eq:ExpectedGHEScLowerBound}),(\ref{eq:ExpectedGHEScUpperBound}),(\ref{eq:GIGVarSCLowerBound}),(\ref{eq:GIGVarSCUpperBound})\), so that $M_{Z_{\epsilon}}^{(1)} \sim \sqrt{\epsilon}$ and $M_{Z_{\epsilon}}^{(2)} \sim \epsilon\sqrt{\epsilon},$ for all $\delta \neq 0$.

    \section{Code repository}
    The code used for all the simulations it his paper can be found at \url{https://github.com/marcostapiac/PyLevy}.

    \section{Acknowledgements}
    We wish to thank the anonymous reviewers for their careful reading of our original submission and their useful comments. In particular, we thank reviewer I for suggesting the direct verification of requirement~(III) from~\cite[Theorem~V.19]{Pollard_1984} in Theorem~\ref{theorem:NVMConvergenceGaussian}.
\bibliography{bibliography.bib}

\begin{thebibliography}{10}

\bibitem{NonLinearGaussianInference}
D.~Alspach and H.~Sorenson.
\newblock Nonlinear {Bayesian} estimation using {Gaussian} sum approximations.
\newblock {\em IEEE Trans. Automatic and Control}, 17(4):439--448, 1972.

\bibitem{SmallJumps}
S.~Asmussen and J.~Rosinski.
\newblock Approximations of small jumps of {L{\'e}vy} processes with a view
  towards simulation.
\newblock {\em Journal of Applied Probability}, 38(2):482--493, 2001.

\bibitem{LevyInFinance}
O.E. Barndorff-Nielsen and N.~Shephard.
\newblock Modelling by {L{\'e}vy} processess for financial econometrics.
\newblock In O.E. Barndorff-Nielsen, S.I. Resnick, and T.~Mikosch, editors,
  {\em L{\'e}vy processes: {Theory} and applications}, pages 283--318.
  Birkh{\"a}user, Boston, MA, 2001.

\bibitem{NormalModifiedStable}
{Barndorff-Nielsen, O.E. and Shephard, N.}
\newblock {Normal modified stable processes}.
\newblock {Department of Economics Discussion Paper Series, University of
  Oxford,
  \texttt{ora.ox.ac.uk/objects/uuid:6dffc465-1250-4360-92df-2682eaa2adde}},
  June 2001.

\bibitem{LevyBasics}
{Barndorff-Nielsen, O.E. and Shephard, N.}
\newblock {Basics of Lévy Processes}.
\newblock Department of Economics Discussion Paper Series, University of
  Oxford,
  \texttt{ora.ox.ac.uk/objects/uuid:8787765a-1d95-45cd-97d8-930fe8816d97}, Jun
  2012.

\bibitem{UpperIncomplete}
H.~Bateman and A.~Erd{\'e}lyi.
\newblock {\em Higher Transcendental Functions}, volume~2.
\newblock Dover Publications, New York, NY, 2006.

\bibitem{BlackScholesModel}
F.~Black and M.~Scholes.
\newblock The pricing of options and corporate liabilities.
\newblock {\em Journal of Political Economy}, 81(3):637--654, 1973.

\bibitem{BondessonSim}
L.~Bondesson.
\newblock On simulation from infinitely divisible distributions.
\newblock {\em Advances in Applied Probability}, 14(4):855--869, 1982.

\bibitem{BoxMuellerMethod}
G.E.P. Box and M.E. Muller.
\newblock A note on the generation of random normal deviates.
\newblock {\em The Annals of Mathematical Statistics}, 29(2):610--611, 1958.

\bibitem{OverviewGodsill}
O.~Capp\'{e}, S.J. Godsill, and E.~Moulines.
\newblock An overview of existing methods and recent advances in sequential
  {Monte Carlo}.
\newblock {\em Proceedings of the IEEE}, 95(5):899--924, 2007.

\bibitem{TotalVariationGaussian}
A.~Carpentier, C.~Duval, and E.~Mariucci.
\newblock Total variation distance for discretely observed {Lévy}
  processes:{A} {Gaussian} approximation of the small jumps.
\newblock {\em Ann. Inst. H. Poincar\'{e}, Probab. Statist.}, 57(2):901--939,
  2021.

\bibitem{FinancialModellingCont}
R.~Cont and P.~Tankov.
\newblock {\em Financial modelling with jump processes}.
\newblock CRC Press, New York, NY, 2004.

\bibitem{DeligiannidisTruncation}
G.~Deligiannidis, S.~Maurer, and M.V. Tretyakov.
\newblock Random walk algorithm for the {Dirichlet} problem for parabolic
  integro-differential equation.
\newblock {\em BIT Numerical Mathematics}, 61(4):1223--1269, 2021.

\bibitem{BoundsSmallJumps}
E.H.A. Dia.
\newblock Error bounds for small jumps of {L{\'e}vy} processes.
\newblock {\em Advances in Applied Probability}, 45(1):86--105, 2013.

\bibitem{GIGinFinance}
E.~Eberlein.
\newblock Application of generalized hyperbolic {L{\'e}vy} motions to finance.
\newblock In O.E. Barndorff-Nielsen, S.I. Resnick, and T.~Mikosch, editors,
  {\em L{\'e}vy Processes: Theory and Applications}, pages 319--336.
  Birkh{\"a}user Boston, Boston, MA, 2001.

\bibitem{JumpTypeProcesses}
E.~Eberlein.
\newblock Jump-type {L{\'e}vy} processes.
\newblock In {\em Handbook of Financial Time Series}, pages 439--455. Springer,
  Berlin Heidelberg, 2009.

\bibitem{GIGHydrologic}
S.~El~Adlouni, F.~Chebana, and B.~Bob{\'e}e.
\newblock Generalized extreme value versus {Halphen} system: {Exploratory}
  study.
\newblock {\em Journal of Hydrologic Engineering}, 15(2):79--89, 2010.

\bibitem{FergussonKlassSim}
T.~Ferguson and M.~Klass.
\newblock A representation of independent increment processes without
  {Gaussian} components.
\newblock {\em The Annals of Mathematical Statistics}, 43:1634--1643, 10 1972.

\bibitem{Fournier2009}
N.~Fournier.
\newblock Simulation and approximation of {L{\'e}vy}-driven stochastic
  differential equations.
\newblock {\em ESAIM: Probability and Statistics}, 15:233--248, 2011.

\bibitem{LevyObjectTracking}
R.~Gan, B.I. Ahmad, and S.J. Godsill.
\newblock Lévy state-space models for tracking and intent prediction of highly
  maneuverable objects.
\newblock {\em {IEEE Transactions on Aerospace and Electronic Systems}},
  57(4):2021--2038, 2021.

\bibitem{GIGLevy}
S.J. Godsill and Y.~Kındap.
\newblock Point process simulation of generalised inverse {Gaussian} processes
  and estimation of the {Jaeger} integral, 2022.

\bibitem{LevyStateSpaceModel}
S.J. Godsill, M.~Riabiz, and I.~Kontoyiannis.
\newblock The {L\'evy} state space model, December 2019.
\newblock arXiv e-prints [math.PR] 1912.12524.

\bibitem{HighFrequencyReturns}
J.~Murphy H.L.~Christensen and S.J. Godsill.
\newblock Forecasting high-frequency futures returns using online {Langevin}
  dynamics.
\newblock {\em IEEE Journal of Selected Topics in Signal Processing},
  6(4):366--380, 2012.

\bibitem{GIGNeuralActivity}
S.~Iyengar and Q.~Liao.
\newblock Modeling neural activity using the generalized inverse {Gaussian}
  distribution.
\newblock {\em Biological Cybernetics}, 77(4):289--295, 1997.

\bibitem{ModernProbabilityFoundations}
O.~Kallenberg.
\newblock {\em Foundations of modern probability}.
\newblock Springer, 1997.

\bibitem{KalmanFilterOriginal}
R.E. Kalman.
\newblock A new approach to linear filtering and prediction problems.
\newblock {\em Transactions of the ASME -- Journal of Basic Engineering},
  82(Series D):35--45, 1960.

\bibitem{KantasSequentialStateSpace}
N.~Kantas, A.~Doucet, S.S. Singh, J.~Maciejowski, and N.~Chopin.
\newblock On particle methods for parameter estimation in state-space models.
\newblock {\em Statistical Science}, 30(3), August 2015.

\bibitem{KhintchineSim}
A.~Khintchine.
\newblock Zur theorie der unbeschr\"ankt teilbaren verteilungsgesetze.
\newblock {\em Rec. Math. [Mat. Sbornik] N.S.}, 2:79--119, 1937.

\bibitem{TSDistAndProcesses}
U.~K{\"u}chler and S.~Tappe.
\newblock Tempered stable distributions and processes.
\newblock {\em Stochastic Processes and their Applications},
  123(12):4256--4293, 2013.

\bibitem{KummerSeries}
E.E. Kummer.
\newblock De integralibus quibusdam definitis et seriebus infinitis.
\newblock {\em Journal f\"{u}r die reine und angewandte Mathematik},
  1837(17):228--242, 1837.

\bibitem{Lemke_Godsill_2015}
T.~Lemke and S.J. Godsill.
\newblock Inference for models with asymmetric {$\alpha$}-stable noise
  processes.
\newblock In {\em {Unobserved Components and Time Series Econometrics}}, pages
  190--217. Oxford University Press, 11 2015.

\bibitem{VarianceGammaOptionPricing}
D.B. Madan, P.P Carr, and E.C Chang.
\newblock The variance gamma process and option pricing.
\newblock {\em Review of Finance}, 2(1):79--105, April 1998.

\bibitem{Landis2013PhylogeneticAU}
J.G.~Schraiber M.J.~Landis and M.~Liang.
\newblock Phylogenetic analysis using {Lévy} processes: {Finding} jumps in the
  evolution of continuous traits.
\newblock {\em Systematic biology}, 62-2:193--204, 2013.

\bibitem{Pollard_1984}
D.~Pollard.
\newblock {\em Convergence of Stochastic Processes}.
\newblock Springer, New York, NY, 1984.

\bibitem{CharacteristicLevyIntegral}
M.~Riabiz, T.~Ardeshiri, I.~Kontoyiannis, and S.J. Godsill.
\newblock Non-asymptotic {Gaussian} approximation for inference with stable
  noise.
\newblock {\em IEEE Transactions on Information Theory}, 66(8):4966--4991,
  2020.

\bibitem{GenShotNoiseStochIntegral}
J.~Rosinski.
\newblock On a class of infinitely divisible processes represented as mixtures
  of gaussian processes.
\newblock In S.~Cambanis, G.~Samorodnitsky, and M.S. Taqqu, editors, {\em
  Stable Processes and Related Topics: A Selection of Papers from the
  Mathematical Sciences Institute Workshop}, pages 27--41. Birkh{\"a}user,
  Boston, MA, 1991.

\bibitem{GeneralisedShotNoise}
J.~Rosinski.
\newblock Series representations of {L{\'e}vy} processes from the perspective
  of point processes.
\newblock In Ole~E. Barndorff-Nielsen, Sidney~I. Resnick, and Thomas Mikosch,
  editors, {\em L{\'e}vy Processes: Theory and Applications}, pages 401--415.
  Birkh{\"a}user, Boston, MA, 2001.

\bibitem{TSLevyFlight}
J.~Rosinski.
\newblock Tempering stable processes.
\newblock {\em Stochastic Processes and their Applications}, 117(6):677--707,
  2007.

\bibitem{RUSSO2009521}
T.~Russo, P.~Baldi, A.~Parisi, G.~Magnifico, S.~Mariani, and S.~Cataudella.
\newblock {Lévy} processes and stochastic {von Bertalanffy} models of growth,
  with application to fish population analysis.
\newblock {\em Journal of Theoretical Biology}, 258(4):521--529, 2009.

\bibitem{StableNonGaussianProcesses}
G.~Samorodnitsky and M.S. Taqqu.
\newblock {\em Stable non-Gaussian random processes: {Stochastic} models with
  infinite variance: {Stochastic} modeling}.
\newblock Routledge, New York, NY, 2017.

\bibitem{InterestRatesVG}
A.M. Udoye and G.O.S. Ekhaguere.
\newblock Sensitivity analysis of interest rate derivatives in a variance gamma
  markets.
\newblock {\em Palestine Journal of Mathematics}, 11(2):159--176, 2022.

\bibitem{TotalVariationGaussian2}
B.~Vlad and Q.~Yifeng.
\newblock Total variation distance between a jump-equation and its {Gaussian}
  approximation.
\newblock {\em Stochastics and Partial Differential Equations: Analysis and
  Computations}, 10(3):1211--1260, 2022.

\bibitem{MSCOxford}
M.~Winkel.
\newblock {MS3b/MScMCF Lévy} processes and finance.
\newblock \url{www.stats.ox.ac.uk/~winkel/ms3b10.pdf}, January 2010.

\bibitem{CompensatorBath}
R.L. Wolpert and K.~Ickstadt.
\newblock Simulation of {L{\'e}vy} random fields.
\newblock In D.~Dey, P.~M{\"u}ller, and D.~Sinha, editors, {\em Practical
  Nonparametric and Semiparametric Bayesian Statistics}, pages 227--242.
  Springer, New York, NY, 1998.

\bibitem{LevyNitrate}
D.B Woodard, R.L. Wolpert, and M.A. O’Connell.
\newblock Spatial inference of nitrate concentrates in groundwater.
\newblock {\em Journal of Agricultural, Biological and Environmental
  Statistics}, 15:209–227, 2010.

\end{thebibliography}
\bibliographystyle{plain}

\end{document}